\documentclass[12pt,a4paper]{article}
\usepackage{mathpaper}

\theoremstyle{mystyle_thm}
\newtheorem{theorem}{Theorem}[section]
\newtheorem{lemma}[theorem]{Lemma}
\newtheorem{proposition}[theorem]{Proposition}
\newtheorem{corollary}[theorem]{Corollary}

\newtheorem*{propositionOhne}{Proposition}

\theoremstyle{mystyle_def}
\newtheorem{definition}[theorem]{Definition}
\newtheorem{example}[theorem]{Example}

\theoremstyle{mystyle_rmk}
\newtheorem{remark}[theorem]{Remark}

\newtheorem{notation}[theorem]{Notation}

\newtheorem{SA}[theorem]{Standing Assumption}

\newtheorem*{conventions}{Clarifications and Conventions}

\usepackage{bbm}

\DeclareMathOperator{\Esp}{E}
\DeclareMathOperator{\Prob}{P}

\DeclareMathOperator{\IR}{\mathbb{R}}
\DeclareMathOperator{\IN}{\mathbb{N}}
\DeclareMathOperator{\bF}{\mathcal{F}}
\DeclareMathOperator{\dom}{dom}

\DeclareMathOperator{\card}{card}

\newcommand{\ext}[0]{0}
\newcommand{\mc}[0]{\mathcal}
\newcommand{\wh}[0]{\widehat}
\newcommand{\wt}[0]{\widetilde}
\newcommand{\msf}[0]{\mathsf}

\newcommand{\s}[0]{ s}
\newcommand{\m}[0]{ m}
\newcommand{\q}[0]{ q}
\newcommand{\Lop}[0]{\mathrm L}
\newcommand{\D}[0]{\mathrm D}

\newcommand{\mX}[0]{\msf X}
\newcommand{\mW}[0]{\msf W}
\newcommand{\mY}[0]{\msf Y}
\newcommand{\bv}[0]{\mathbf{v}}
\newcommand{\by}[0]{\mathbf{y}}
\newcommand{\bx}[0]{\mathbf{x}}

\newcommand{\ba}[0]{\mathbf{a}}
\newcommand{\bb}[0]{\mathbf{b}}

\newcommand{\xlra}[2]{ \xrightarrow[#2]{#1} }

\newcommand{\rd}{\mathrm{d}}
\newcommand{\vd}{\,\mathrm{d}}
\newcommand{\process}[1]{(#1)_{t\ge 0}}

\newcommand{\indicB}[1]{\mathbbm{1}_{\{#1\}}}
\newcommand{\loct}[3]{L^{#2}_{#3}(#1)}

\newcommand{\braces}[1]{ \left({#1}\right) } 
\newcommand{\sqbraces}[1]{ \left[{#1}\right] } 
\newcommand{\cubraces}[1]{ \left\{{#1}\right\}} 
\newcommand{\xnorm}[2]{ \left\|{#1}\right\|_{#2}} 
\newcommand{\abs}[1]{ \left|{#1}\right|} 
 
\newcommand{\qv}[1]{ \left\langle {#1} \right\rangle }
\begin{document}
	
	\title{General Diffusions on the Star Graph\\ as Time-Changed Walsh Brownian Motion}
	
	\author{Alexis Anagnostakis\thanks{Universit\'e de Lorraine, CNRS, IECL, F-54000 Nancy, France} \\ \small \hemail{alexis.anagnostakis@univ-lorraine.fr}}
	
	\date{Version: \today}
	
	\maketitle
	
	\begin{abstract}
We establish the representation of general regular diffusions on star-shaped graphs as time-changed Walsh Brownian motions. These are regular continuous Markov processes described locally by a family generalized second order differential operators defined on every edge and a gluing condition at the junction vertex. 
This allows us to prove two additional results: 
(i) A representation of diffusions with sticky gluing conditions as time-changes of diffusions governed by the same differential operators but with non-sticky gluing conditions.
(ii) An occupation times formula for such diffusions, analogous to the classical It\^o--McKean formula for one-dimensional diffusions.

Additionally, we prove two results of independent interest. First, conditions under which a diffusion on the star graph is Feller and Feller--Dynkin, extending classical results for one-dimensional diffusions. Second, the existence uniqueness of solutions to the Dirichlet problem on the unit disk of the star graph for a general diffusion operator and explicit expressions for its solution.
	\end{abstract}

	\afabs{Keywords and phrases}{Diffusion on network; generalized second order operator; 
		occupation times formula; sticky lateral condition; It\^o--McKean time change; Dirichlet problem; Feller--Dynkin process.} 
	

%% New Section
\section{Introduction}
\label{sec_introduction}

The study of diffusions on metric graphs and networks has been a recurring topic of interest in probability theory since the pioneering works of Feller (one may consult the seminal
works~\cite{freidlin2000diffusion,freidlin1993diffusion,lumer1980connecting}). 
These processes arise in diverse contexts, 
as noted in~\cite{weber2001occupation} with updated references: 

\begin{quotation}
	\emph{Such processes appear as models, for example, of electrical networks, vibrating elastic nets, nerve impulse propagation, or the movement of nutrients in the root system of a plant (see~\cite{frank1984random,friedman1994diffusionsnetworks}).}
\end{quotation}

Among these, diffusions on \textit{star-shaped graphs}---comprising multiple outward-pointing edges connected at a single junction vertex---have attracted significant attention due to their simple topology and their role as building blocks for more complex processes. Such graphs are often referred to as \textit{spiders}, and the corresponding processes are termed \textit{diffusion spiders} in the literature. The simplest example is the \textit{Walsh Brownian motion} on the star graph (or \textit{Brownian spider}), which behaves like a standard Brownian motion on each edge and randomly selects among edges upon departing the junction vertex according to fixed probabilities, known as \textit{bias parameters}.

From a purely mathematical perspective, star-shaped graphs serve as a natural setting for exploring the interaction between local diffusion behavior and global topological constraints. The central challenge lies in characterizing the behavior of diffusions at the junction vertex, where the process may exhibit sticky behavior---the process spends a positive amount of time at the vertex.

Recent literature has seen a surge of interest in diffusions on star-shaped graphs, with several key contributions. 
For example, in~\cite{freidlin2000diffusion}, the authors derive an It\^o  formula for diffusions that are absolutely continuous on each edge under non-sticky lateral conditions. 
In~\cite{lempa2024diffusion}, the resolvent density of a diffusion spider is computed.
This is used to characterize excessive functions of the diffusion and to solve optimal stopping problems. 
The joint moments of occupation times of edges of a diffusion spider are investigated in~\cite{salminen2024occupation}, while the diameter of the optimal stopped spider process is studied in~\cite{bednarz2024diameter}. 
In~\cite{bobrowski2024snapping}, the Walsh Brownian motion is approximated by snapping out Brownian motions
with high permeability coefficients. 
In~\cite{berry2024sticky}, the authors study diffusions on star-shaped graphs that are absolutely continuous on each edge and feature a sticky gluing condition at the junction vertex. They demonstrate that such a diffusion can be expressed as a time-change of a non-sticky diffusion, generalizing results in~\cite{EngPes,Sal2017} for one-dimensional diffusions. Building on this result, they propose a sticky version of the It\^o  formula introduced in~\cite{freidlin2000diffusion}, referred to as the sticky Freidlin--Sheu--It\^o  formula. 

In parallel, there has been growing interest in stochastic control problems for diffusions
and mean field games on graphs and networks. 
To cite a few recent contributions,~\cite{achdou2019class_infinite,barles2024nonlocal,berry2024stationary,camilli2024continuousdependence,ohavi2023comparison}. 

\medskip

In this paper, we study \emph{general diffusions}---that is, regular continuous strong Markov processes---on star-shaped graphs. As shown in~\cite{freidlin1993diffusion}, these processes are characterized by:
\begin{itemize}
	\item A family of generalized second-order differential operators describing the behavior on each edge;
	\item A (sticky) gluing condition specifying the behavior at the junction vertex $\mathbf{v}$;
	\item A set of boundary conditions that completely determine the boundary behavior.
\end{itemize}

In this paper, we show that any such diffusion $\mX$ can be represented as a re-scaled, time-changed Walsh Brownian motion with the same bias parameters as $\mX$. 
This allow us to prove two additional results. 
First, a representation of $\mX $ as time-change of a diffusion that is non-sticky at $ \bv$ and whose law matches $\mX$ locally on every each edge.
This generalizes~\cite[Propositions~4.1 and Theorem~4.4]{berry2024sticky} from diffusions determined by classical second order operators to those defined by generalized second order operators.
Second, we derive an occupation times formula for $\mX$, linking occupation and local time at $\bv $ in a manner analogous to one-dimensional sticky diffusions.

The main difference between the time-change characterization and the homogeneous occupation times formula for diffusions on the star graph, compared to their one-dimensional counterparts, is the presence of an additional integration in time to account for  edge-occupation by the process.

Additionally, we prove two auxiliary results of independent interest. First, we provide  conditions under which a diffusion on the star graph is Feller and Feller--Dynkin, extending classical results for one-dimensional diffusions. Second, we establish the existence and uniqueness of solutions to the Dirichlet problem on the unit disk of the star graph, along with an explicit expression for the corresponding Green function. These results play a crucial role in our derivations: the first justifies the use of Dynkin's formula and Dynkin's operator, while the second enables computations at the junction vertex that are compatible with the gluing condition.

The main challenge in this work lies in the irregularity of the functional spaces associated with diffusions issued from generalized second order operators. 
Unlike classical diffusions, where the domain of the generator consists of smooth functions, the functions in our setting may lack sufficient regularity. 
This precludes the use of second order expansions, as employed in~\cite{Sal2017} for one-dimensional diffusions and in~\cite{berry2024sticky} for non-sticky diffusions on star graphs. 
Instead, our approach is inspired by the proof of~\cite[Theorem~IV.47.1]{RogWilV2} in the presence of reflecting boundaries.

\medskip

Expressing processes as time-changes of simpler ones is a powerful and versatile tool in the theory of continuous processes. 
In one-dimensional settings, this is typically achieved through the Dambis--Dubins--Schwarz theorem for local martingales and the It\^o--McKean time-change for regular diffusions on natural scale. 
These characterizations have been applied in numerous settings, leading to significant advancements in stochastic analysis. 
Some results proven via time-change are the following:
\begin{itemize}
	\item The celebrated occupation times formula for diffusions on natural scale of It\^o  and McKean~\cite[Section~5.4]{ItoMcKean96}.
	
	\item The path-space description of one-dimensional sticky diffusions~\cite{EngPes,Sal2017,warren2015stickyflows}
	and explicit versions of stochastic calculus results in the presence of stickiness
	(It\^o formula, Girsanov theorem, Freidlin--Sheu--It\^o  formula)~\cite{Anagnostakis2022,berry2024sticky,Sal2017}.
	
	\item Definition of Diffusions: The definition of sticky L\'evy process in~\cite{ramirez2024stickyLevy} and characterization of skew Brownian motion in~\cite{harrisson1981onskew}.
	
	\item Approximations and Convergence: The approximation one-dimensional diffusions by birth-and-death processes, particularly in cases where the diffusion coefficients are irregular or singular (see~\cite[Chapter~6]{ethier2005markov}).
	The convergence convergence rates for the approximation schemes in~\cite{anagnostakis2023general,AnkKruUru2} of one-dimensional diffusions.
	The convergence in law of diffusions if both weak convergence of their speed measures and pointwise convergence of their scale functions occur~\cite{brooks1982weak}. 
	The convergence of the frozen random walk in~\cite{Ami} to the sticky Brownian motion.
\end{itemize}

The present work is motivated by the simulation of diffusions on graphs. 
In a further work, we use the time-change characterization to derive sufficient conditions for regularity estimates of a diffusion on a star graph. 
This allows proof of convergence for random walk approximations, as is done in~\cite{anagnostakis2023general} for one-dimensional diffusions. 

\subsection*{Paper Outline}

The paper is organized as follows. Section~\ref{sec_setting} introduces the necessary mathematical framework and presents the main results of the paper. Section~\ref{sec_char} establishes preliminary results, including (i) the role of the gluing condition in determining the infinitesimal behavior of the process at the junction vertex $\bv$, (ii) a construction of (sticky) Walsh Brownian motion, and (iii) the behavior of diffusion laws on $\Gamma$ under a change of scale. Section~\ref{sec_first_char} proves a Dambis--Dubins--Schwarz-type representation of a diffusion on $\Gamma$ as a time-changed Walsh Brownian motion. Section~\ref{sec_further} characterizes the local time of NSE diffusions and derives the corresponding It\^o formula. Section~\ref{sec_main_proof} provides proofs for the main results (Theorems~\ref{thm_main} and \ref{thm_main2}). 
Section~\ref{sec_occupation_II} derives an occupation times formula for NSE diffusions on the star graph, akin to the formula for one-dimensional diffusions in \cite[Section~5.4]{ItoMcKean96}.

Appendix~\ref{app_boundary} presents the boundary classification problem on the star-graph. 
Appendices~\ref{app_Fprocess} and~\ref{app_FD} explore conditions for a diffusion on the star to be Feller and Feller--Dynkin. The (sticky) Walsh Brownian motion, along with all processes studied in this paper, satisfies this condition. 
Appendix~\ref{app_strict_monotony} establishes the strictly increasing character of the quadratic variation of the distance to origin process of a regular diffusion 
on $\Gamma $. 
This will justify the time-change manipulations in our derivations. 
Appendix~\ref{app_Dirichlet} focuses on the Dirichlet problem for generalized second order differential operators on the unit disk of the star graph. From this analysis, we infer the role of the stickiness parameter in the gluing condition.

%% New Section
\section{First definitions and main results}
\label{sec_setting}

Following~\cite{mugnolo2019actually} we define star-shaped graphs by disjoint joining 
of intervals modulo an equivalence relation that expresses the geometric constraint at the junction vertex. 

More precisely, let 
\begin{equation}
	\label{eq_Gamma}
	\msf \Gamma  := \bigsqcup_{e\in E} O_e,
\end{equation} 
where for each $e\in E $, $O_e $ is an interval with a closed endpoint at $0$ and
another endpoint at $l_e \in (0,\infty] $.
Let $\sim $ be the equivalence relation defined on $\Gamma $ as
\begin{equation}
	\label{eq_def_equiv}
	\begin{aligned}
		(e,x)&\sim (e',y)
		& \text{if and only if}&
		& (e,x)=(e',y)&
		& \text{or}&
		& x=y={}&0,
	\end{aligned}
\end{equation}
for all $(e,x),(e',y)\in \Gamma $.
Then, the quotien set $\Gamma = \msf \Gamma/\sim $ is a star-shaped metric graph of edge-lengths $e\mapsto l_e $. 

With this construction, every element $\bx$ of $\Gamma $ admits the representation $(e,x)\in E \times O_e $, where $e$ is the edge on which $\bx $ is located and $x$ is the Euclidean distance-to-origin. 

We equip $\Gamma$ with the metric $d$, defined as
\[
d\big((e, x), (e', y)\big) =
\begin{cases}
	|x - y|, & \text{if } e = e', \\
	x + y, & \text{if } e \neq e'.
\end{cases}
\]
The space $(\Gamma, d)$ is a \textit{locally compact Polish space}. Throughout the paper, we assume the topology on $\Gamma$ that is induced by $d$ and denote by $B(\bv, \delta)$ the open ball centered at the junction vertex $\bv$ of radius $\delta > 0$. 

\medskip

We now present the notion of function on $\Gamma $ and define some functional spaces of interest. 
Following the representation $\Gamma = \msf \Gamma /\sim $, we may represent a function $f$ on 
$\Gamma $ as a direct sum of functions on $[0,\infty) $ that respects the structure~\eqref{eq_def_equiv}, i.e., 
\begin{equation}
	\label{eq_txt_function_representation}
	\begin{aligned}
		f &= \bigoplus_{e\in E} f_e,
		& f_e:{}  & O_e \mapsto \IR,
		& f_{e}(0)&=f_{e'}(0),
		& \forall\,& e,e'\in E,
	\end{aligned}
\end{equation}
so that $f(e,x) = f_e(x) $, for all $e\in E $ and $x\in O_e $, and so that
$f(e,0)= f(e',0) $, for all $e,e'\in E $. 
We denote with $C(\Gamma) $ the space of continuous real-valued functions on $\Gamma $,
with $C_b(\Gamma) $ the space of bounded continuous functions, and
with $C_0(\Gamma) $ the space of continuous functions that vanish at infinity.
These are defined as
\begin{align}
	C(\Gamma) &:= \braces{ f : \Gamma \to \mathbb{R} :\; f(e,\cdot) \in C(O_e), \; \forall\, e \in E },\\
	C_b(\Gamma) &:= \braces{ f \in C(\Gamma) : \exists M > 0 \text{ such that } |f(e, x)| \leq M, \; \forall\, e \in E, \; \forall\, x \in O_e },\\
	C_0(\Gamma) &:= \braces{ f \in C(\Gamma) : \lim_{x \to \infty} f(e, x) = 0, \; \forall\, e \in E }.
\end{align}

A diffusion on $\Gamma $ is a continuous strong Markov process on $\Gamma $. 
Assume such process $\mX $ defined on the probability space $\mc P_{\bx} := (\Omega,\process{\bF_t},\Prob_{\bx}) $ where $\mX_0 = \bx $, $\Prob_{\bx} $-almost surely and the filtration $\process{\bF_t} $ satisfies the usual conditions:
\begin{itemize}
	\item (right-continuity) $\bigcap_{s \ge t} \bF_s = \bF_t $,
	\item (completeness) $\mc F_0 $ contains all $\Prob_{\bx}$-negligible sets.
\end{itemize}
These conditions will be assumed on all filtrations throughout this paper. 

\medskip

A \emph{general diffusion} on $\Gamma$ is a strong Markov process on $\Gamma $ that is regular in the sense that $\Prob_{\bx} (T_{\by}< \infty) >0 $, for all $\bx \in \Gamma^{\circ}$ and $\by \in \Gamma $, where $T_{\by}$ is the hitting time of $\by $ defined as
$T_{\by} := \inf\{t \ge 0 : \mX_t = \by\} $.
By \cite[Theorem~3.1]{freidlin1993diffusion} a general diffusion on
$\Gamma $ is defined by:
\begin{xenumerate}{A}

	\item \label{item_A1} A family of generalized second order differential operators (in the sense of~\cite{Fel55})
	of the form $\frac{1}{2} \D_{m_e}\D_{s_e} $, where for all $e\in E $,
	$s_e $ is a continuous increasing function and $m_e $ a locally finite positive measure on $(0,l_e) $. 
	The pairs \((\s_e, \m_e;\, e \in E)\) describe the local behavior of the process on each edge \(e \in E\).
	In particular (see \cite[Definitions~VII.3.3 and~VII.3.7]{RevYor}), for such a process \(\mX\) defined on the probability space \((\Omega, \process{\bF_t}, \Prob_{\bx})\) such that \(\Prob_{\bx}\)-a.s. \(\mX_0 = \bx\), it holds that
	\begin{equation}
		\label{eq_scalespeed}
		\Prob_{\bx} \braces{T_{\bb}<T_{\ba}} = \frac{\s_e(x)-\s_e(a)}{\s_e(b)-\s_e(a)}
		\quad \text{and}
		\quad 
		\Esp_{\bx} \sqbraces{ T_{\bb} \wedge T_{\ba} }
		= 2 \int_{a}^{b} G^{(e)}_{a,b}(x,y) \m_e(\rd y),
	\end{equation}
	for all $a,x,b \in O_e $ with $a<x<b$, where 
	\begin{equation}
		G^{(e)}_{a,b}(x,y) := \frac{\braces{\s_e(x \wedge y)-\s_e(a)}\braces{\s_e(b)-\s_e(x \vee y)}}{\s_e(b)-\s_e(a)},
		\quad \forall\, a<x,y<b.
	\end{equation}

	\item \label{item_A2} A gluing condition at the junction vertex:
	\begin{equation}
		\label{eq_def_lateral}
		\sum_{e' \in E} \beta_{e'} f'_{e'}(0) = \rho \Lop_{e} f_e(0), \quad \forall\, e \in E,
	\end{equation}
	where $\rho \ge 0 $ and $(\beta_e)_{e\in E} $ satisfy:
	$\sum_{e'\in E} \beta_{e'} = 1 $ and $\beta_e >0 $, for all $e\in E $.

	\item \label{item_A3} A family of lateral conditions for every regular boundary $(e,l_e) $:
	\begin{equation}
		f'(e,l_e) = \rho_{e} \Lop_{e} f_e(l_e). 
	\end{equation}
	See Proposition~\ref{prop_boundary_classification} for a definition of regular boundaries and a presentation of the boundary classification problem for diffusions on $\Gamma $. 
\end{xenumerate}

The term general is used to refer to the (potentially)
generalized nature of the operators $(\Lop_e)_{e\in E} $.
If $\rho = 0 $, time the process spends at $\bv $ is almost surely of zero Lebesgue measure.
In that case, we say the process is non-sticky at $\bv $ and the condition reads:
\begin{equation}
	\label{eq_def_lateral2}
	\sum_{e' \in E} \beta_{e'} f'_{e'}(0) = 0.
\end{equation}
If the diffusion defined above satisfies $\s_e(x)=x $, for all $x \in O_e $ and $e\in E $ is called \textit{on Natural Scale on Edges} (or \textit{NSE}). 

According to the representation $\Gamma = \msf \Gamma/\sim$, a process $\mX$ on $\Gamma$ can be factorized into a pair of processes $(J, X)$, where $J:\IR_+ \mapsto E $ and $X:\IR_+ \to \IR_+ $.
For all $t\ge 0 $, $J(t) $ is the edge on which $\mX_t $ is located and
$X_t $ is the distance to origin of $\mX_t $, i.e., $X_t = d(\mX_t,\bv) $.

\medskip

We now comment on the Feller and Feller--Dynkin character of such processes.
Assume $\mX $ is defined on the probability space $\mc P_{\bx} = (\Omega,\process{\bF_t},\Prob_{\bx}) $,
such that $\mX_0 = \bx $, $\Prob_{\bx} $-almost surely. 
The \emph{semigroup} of $\mX $ is the family of functionals $\process{P_t} $ on
$C(\Gamma) $ defined as 
\begin{equation}
		P_t f(\bx) := \Esp_{\bx} \sqbraces{ f(\mX_t) },
		\quad \text{for all } f\in C(\Gamma),
		\text{ } \bx \in \Gamma,
		\text{ and } t\ge 0.
\end{equation}
Following \cite[Section~III.6]{RogWilV1}, the diffusion $\mX $ is called \textit{Feller} if it  satisfies
$P_t \braces{C_b(\Gamma)} \subset C_b(\Gamma) $ and \textit{Feller--Dynkin} if it satisfies $P_t \braces{C_0(\Gamma)} \subset C_0(\Gamma) $.  

As in the case of one-dimensional diffusions, whether a diffusion on $\Gamma$ is Feller--Dynkin depends entirely on its boundary behavior (see \cite[Theorem~1.1]{criens2023feller} for the one-dimensional case). This is evident from the following proposition. The proof is provided in Appendix~\ref{app_FD}.

\begin{proposition}
	\label{prop_FD}
	A regular diffusion on $\Gamma$ is Feller--Dynkin if and only if all open boundaries are natural.
\end{proposition}

Throughout this paper we make the following assumption on the boundary behavior of the process $\mX $. 
This simplifies our analysis significantly as we are not required to handle boundary behaviors at regular boundaries. 
Moreover, by virtue of Proposition~\ref{prop_FD}, such diffusions are Feller--Dynkin
which allows to use specific results adapted for such processes. 

\begin{SA}
	\label{assumption_regularity}
	The diffusion has only natural boundaries and therefore, by Proposition~\ref{prop_FD}, it is Feller--Dynkin. 
\end{SA}

This assumption has several important implications:
\begin{enumerate}
\item
The scale function \(s\) is bounded in the vicinity of \(\bv\). 
Both the scale \(s\) and inverse scale \(q\) are continuous on \(\Gamma^{\circ}\) and \((s(\Gamma))^{\circ}\), respectively.
Since \((s_e;\, e \in E)\) act via their right-derivatives, we assume \(s_e(0) = 0\) for all \(e \in E\) so that $s \in C(\Gamma) $. 
	
	\item Since there are no regular boundaries a general diffusion on $\Gamma $ is fully determined by \ref{item_A1}--\ref{item_A2}. No additional boundary conditions
	like \ref{item_A3} are needed. 
	
	\item By virtue of Proposition~\ref{prop_FD}, since the diffusion is Feller--Dynkin, we work with the \(C_0\)-generator \(\Lop\) and its composites \((\Lop_e;\, e \in E)\). 
	These are the operators defined as 
	\begin{equation}
		\Lop f(e,x) =\Lop_e f(e,x) = \frac{1}{2} \D_{m_e} \D_{s_e} f(e,x), 
		\quad \text{for all } (e,x)\in \Gamma, 
		\text{ and } f\in \dom_{C_0}(\Lop) ,
	\end{equation}
	where
	\begin{align}
		\dom_{C_0}(\Lop) :={} & \bigg\{f = \bigoplus_{e\in E} f_e :\; f_e\in 
			\dom_{C_0}(\Lop_e), \; \forall\, e\in E ;\;\\
			& \quad f_e(0)=f_{e'}(0) \text{ and } \Lop_e f_e(0)= \Lop_{e'} f_{e'}(0),\; \text{for all } e,e'\in E;\;
			\text{and \eqref{eq_def_lateral} holds} \bigg\},
		\\ 
		\dom_{C_0}(\Lop_e) :={} & \cubraces{ f \in C_0([0,\infty)):\; \Lop_e f \in C_0([0,\infty)) }.
	\end{align}
\end{enumerate}

To motivate the general framework we adopt, let us present an example of a diffusion process on \(\Gamma\), defined by a family of non-classical second order differential operators.

\begin{example}
	Consider the general diffusion on 
	\begin{equation}
		\label{eq_gamma_unconstrained}
		\Gamma^{0} := \braces{\bigsqcup_{e\in E} [0,\infty)}/\sim
	\end{equation}
	defined by~\ref{item_A1}--\ref{item_A3}, where the operators \((\Lop_e;\, e \in E)\) are defined as follows:
	\begin{itemize}
		\item For a specific edge $e_1 \in E $, let $\Lop_{e_1} = \frac 12 \D_{\m_{e_1}} \D_{\s_{e_1}} $,  where the pair $(\s_{e_1},\m_{e_1}) $ is defined by
		\begin{equation}
			\s'_{e_1}(x) := \exp \braces{ - \int_{0}^{x} \frac{2b(y)}{ \sigma^{2}(y)} \vd y }, \quad
			\m_{e_1}(\rd x) := \frac{\rd x}{\s'_{e_1}(x)\sigma^{2}(x)},
			\quad \forall\,x >0,
		\end{equation}
		where $b,\sigma $ are the functions defined as $b(x) := (1 - x)$ and $\sigma(x) := \sqrt{x+1} $, for all $x> 0 $.
		\item For another edge  $e_2 \in E $, with $e_2 \not = e_1 $, let $\Lop_{e_2} = \frac 12 \D_{\m_{e_2}} \D_{\s_{e_2}} $, where the pair $(\s_{e_2},\m_{e_2}) $ is defined by
		\begin{equation}
			\s_{e_2}(x) := x, \quad
			\m_{e_2}([a,b]) := F_{C_{1/3}}(b)-F_{C_{1/3}}(a), 
			\quad \forall\,x,a,b>0, 
		\end{equation}
		where $F_{C_{1/3}} $ is the Cantor--Lebesgue function (see \cite[pp. 38, 126]{Stein2005}). 
		\item For all  $e \in E \setminus \{e_1,e_2\} $, let $\Lop_{e} = \frac 12 \D_{\m_e} \D_{\s_e} $, where the pair $(\s_e,\m_e) $ is defined by
		\begin{equation}
			\s_e(x) := x, \quad
			\m_e(\rd x) := \vd x + \delta_{1}(\rd x),
			\quad \forall\,x>0.
		\end{equation}
	\end{itemize}
\end{example}

This diffusion defined above behaves on the edge $e_1 $ like a translated CIR
process; on the edge $e_2 $ like a Brownian motion slowed on the
Cantor set (see \cite[Section~8]{ankirchner2020functional}); and on all remaining edges like a Brownian motion with a sticky threshold located at $1$.
Note that the operators \(\Lop_e\) for \(e \neq e_1\) are not classical second order differential operators. 
Also, since $s_{e_1} $ is not a natural scale, the diffusion is not NSE. 

Let us comment on our presentation in relation to the existing literature.

\begin{conventions}
	\begin{enumerate}
		\item There is some ambiguity in the literature regarding the notion of diffusion. Our definition coincides with those in~\cite{ItoMcKean96, Kal, RogWilV1}, while in~\cite{RevYor}, only the weak Markov property is required.
		
		\item In this paper, we adopt the same scaling convention for the speed measure as in~\cite{Kal} and Rogers and Williams~\cite{RogWilV2}.
		
		\item The notion of \textit{Natural Scale on Edges} (NSE) should not be confounded with the notion of being on natural scale. 
		To illustrate this distinction, consider a skew Brownian motion, which can be seen as a Walsh Brownian motion on a star graph comprised of two edges. 
		By definition, this process is NSE, but if $\beta_e \not= 1/2$, it is not on natural scale (see~\cite{harrisson1981onskew}). 
	\end{enumerate}
\end{conventions}

\subsection{(Sticky) Walsh Brownian motion}
\label{ssec_walsh}

Consider the star-shaped graph $\Gamma^{0} $ defined in~\eqref{eq_gamma_unconstrained}. 
We define the simplest (sticky) Feller--Dynkin diffusion on the star-shaped graph: the (sticky) Walsh Brownian motion on $\Gamma^{0} $ (for historical references, see~\cite{Barlow1989walsh,walsh1978diffusion}). 
This diffusion behaves like a Brownian motion on each edge of $\Gamma^{0}$, with excursions departing from $\bv$ choosing randomly among edges $e \in E$ according to fixed probabilities $(\beta_e;\, e \in E)$. 
Additionally, if $\rho > 0$, the diffusion spends a positive amount of time at $\bv$ upon contact.  
The last two properties are inferred from the forthcoming equation~\eqref{eq_relation_occt_loct} and Corollary~\ref{cor_edge_transition_NSE}.

The \textit{sticky Walsh Brownian motion} on $\Gamma^{0} $ is the general diffusion process defined by:
\begin{enumerate}
	\item The family of operators $(\Lop_e;\; e\in E) $, where
	$\Lop_e = \frac{1}{2}\D_x \D_x $, for all $e\in E $.
	This means that $s_e(x)=x $ and $m_e(\rd x) = \vd x $ for all $x> 0 $ and $e\in E $ which by virtue of Proposition~\ref{prop_boundary_classification}, qualifies all boundaries as natural. Its $C_0$-generator therefore reads:
	\begin{equation}
		\label{eq_generator_W}
		\begin{aligned}
			\Lop_e f_{e}(x) &= \frac{1}{2} f_{e}''(x),
			& \forall\, f&\in \dom_{C_0}(\Lop_e),
			& \forall\, x& \ge 0,
			& \forall\, e& \in E,
		\end{aligned}
	\end{equation}
	where $f',f'' $ are the first and second right-derivatives of $f$ and where
	\begin{equation}
		\label{eq_domain_W}
		\begin{aligned}
			\dom_{C_0}(\Lop_e) &= C_{0}^{2}([0,\infty)),
			& \forall\,& e\in E.
		\end{aligned}
	\end{equation}
	\item The gluing condition at the vertex $\bv $: 
	\begin{equation}
		\label{eq_lateral_W}
		\begin{aligned}
			 \sum_{e''\in E} \beta_{e''} f'_{e''}(0) &= \rho f_{e}''(0),
			& \forall\, e&\in E,
		\end{aligned}
	\end{equation}
\end{enumerate}

When $\rho = 0 $, this process 
is a \textit{(non-sticky) Walsh Brownian motion} on $\Gamma^{0} $. 

\begin{remark}
	The sticky Walsh Brownian motion is a Feller--Dynkin diffusion on $\Gamma^{0} $ (by Proposition~\ref{prop_FD}) and is NSE since $\s_e(x)=x $ and $\m_e(\rd x) =\rd x $, for all $e\in E $ and $x\ge 0 $. 
\end{remark}

To explicitly indicate the dependence of (non-sticky) Walsh Brownian motion on its parameters, we will refer to the process defined above as the $(\beta_e;\, e \in E)$-Walsh Brownian motion on $\Gamma^{0} $ or the Walsh Brownian motion on $\Gamma^{0} $, of bias parameters $(\beta_e;\, e \in E)$. 

\subsection{Main results}
\label{ssec_mainresults}

We now present the main results of this paper. 
Let $\mc P_{\bx} = (\Omega, \process{\bF_t}, \Prob_{\bx})$ be a probability space sufficiently rich to support our analysis.
Since the Walsh Brownian motion is defined on unbounded graph, let us note
$\Gamma$ the star-graph defined in~\eqref{eq_Gamma} and $\Gamma^{0}$ the unbounded
star-graph defined in~\eqref{eq_gamma_unconstrained} so that $\Gamma \subset \Gamma^{0} $. 

Let $s $ be the scale of $\mX $ defined as $s(e,x) := (e,s_e(x)) $, for all $(e,x)\in \Gamma $ and $q $ the inverse scale of $\mX $. 
It is the function $q:\Gamma \mapsto s(\Gamma)$ defined as
\begin{equation}
	\label{eq_q}
	q(e,s_e(x)) := q(e,x), \quad \text{for all } (e,x) \in s(\Gamma). 
\end{equation}

All proofs are deferred to Section~\ref{sec_main_proof}.

\begin{theorem}
	\label{thm_main}
	Let \(\mX\) be the regular diffusion on \(\Gamma\) defined on \(\mc P_x\) by~\ref{item_A1}--\ref{item_A2}. 
	There exists a \((\beta_e;\, e \in E)\)-Walsh Brownian motion \(\mW := (J, R)\) on \(\Gamma^{0}\), adapted to \(\process{\bF_{A(t)}}\), defined on an extension of \(\mc P_x\), such that
	\begin{equation}
		\begin{aligned}
			\mX &= q(\mY), 
			& \text{where} \quad 
			\mY_t &= \mW_{\gamma(t)}, 
			\quad \text{for all } t\ge 0. 
		\end{aligned}
	\end{equation}
	Here, \(q\) is the inverse scale of $\mX$, defined in~\eqref{eq_q}, and the process \(t \to \gamma(t)\) is a time-change (in the sense of \cite[Definition~VI.1.2]{RevYor}), defined as the right-inverse (for the definition of right-inverse, see the forthcoming~\eqref{eq_def_right_inverse}) of
	\begin{equation}
		\label{eq_thm_A_timechange}
		\begin{aligned}
			A(t) &:= \sum_{e \in E} \int_{(0,\infty)} \int_{0}^{t} \indicB{J(s)=e} \vd \loct{R}{y}{s} \, \m^{\mY}_e(\rd y) + \frac{\rho^{\mY}}{2} \loct{R}{0}{t}, 
			& t &\ge 0.
		\end{aligned}
	\end{equation}
	In the above expression, \(\rho^{\mY}\) and \((m^{\mY}_e;\, e \in E)\) are the stickiness parameter and speed measures of the NSE diffusion \(\mY = s(\mX)\) (see the forthcoming Proposition~\ref{prop_change_of_scale}), and \((\loct{R}{y}{t};\; y \in \IR,\, t \ge 0)\) is the local time field of \(R\).
	
	\medskip	
	
	Conversely, let $\mW = (J,R) $ be the $(\beta_e;\, e\in E) $-Walsh Brownian motion
	on $\Gamma^{0} $ defined on $\mc P_x $ such that $ \mW_0 = \bx$, $\Prob_{\bx} $-almost surely.
	Also, let $\Gamma $ an open subgraph of $\Gamma^{0} $, $(\s_e,\m_e;\, e\in E) $ a family of scale and speeds on $\Gamma^{0} $, $\rho \ge 0 $, $t\to A(t) $ be the process defined in~\eqref{eq_thm_A_timechange}, and $t \to \gamma(t) $ be its right-inverse.
	Then, the process $\mX := (q(\mW_{\gamma(t)});\, t\ge 0) $ is a regular diffusion on $\Gamma $, adapted to $\process{\bF_{\gamma(t)}} $, defined  by~\ref{item_A1}--\ref{item_A2}. 
\end{theorem}

We now present a characterization in terms of a diffusion with non-sticky vertex whose generator locally coincides with the one of $\mX $, locally, on each $e\in E $. 
This generalizes~\cite[Proposition~4.1]{berry2024sticky} and  \cite[Theorem~4.4]{berry2024sticky}, which considers only diffusions on $\Gamma^{\ext}$ where $(\Lop_e;\, e \in E)$ is comprised exclusively of classical second-order differential operators. 
A classical second order differential operator 
is one of the form
\begin{equation}
	\begin{aligned}
		\Lop &= b \cdot \D_x + \frac{1}{2} \sigma^{2} \cdot \D_{xx},
	\end{aligned}
\end{equation}
where $b$ is a measurable and $\sigma$ a positive measurable function.

\begin{theorem}
	\label{thm_main2}
	Let $\mX $ be the regular diffusion on $\Gamma $ defined defined on $\mc P_x $, by~\ref{item_A1}--\ref{item_A2}. 
	There exists a regular diffusion $\mX^{\circ}=(I^{\circ},X^{\circ}) $ on $\Gamma $,
	adapted to $\process{\bF_{A_{\rho}(t)}} $, defined 
	on an extension of $\mc P_x $ by~\ref{item_A1} and~\eqref{eq_def_lateral2}, 
	such that
	\begin{equation}
			\mX_t = \mX^{\circ}_{\gamma_{\rho}(t)},
			\quad \text{for all } t\ge 0. 
	\end{equation}
	Here, the process $\gamma $ is a time-change, defined as the right-inverse of
	\begin{equation}
		\label{eq_thm_A_timechange2}
			A_{\rho}(t):= t + \frac{\rho^{\mY}}{2} \loct{Y^{\circ}}{0}{t},
			\quad t\ge 0. 
	\end{equation}   
	with $\mY^{\circ} = (I^{\circ},Y^{\circ}) := \s(\mX^{\circ}) $ and $\rho^{\mY} $
	is the stickiness parameter of the NSE diffusion $\mY := \s(\mX) $, defined in the forthcoming  Proposition~\ref{prop_change_of_scale}, and $(\loct{Y^{\circ}}{y}{t};\; y\in \IR,\, t\ge 0) $ is the local time field of $Y^{\circ}$.
	
	\medskip 
	
	Conversely, let $\mX^{\circ} = (I^{\circ},X^{\circ}) $ be the regular diffusion
	on $\Gamma $ defined on $\mc P_x $, by~\ref{item_A1} and~\eqref{eq_def_lateral2}.
	For $\rho\ge 0 $, let $\gamma_{\rho} $ be the time-change defined in~\eqref{eq_thm_A_timechange2}. 
	Then, the process $\mX := (\mX^{\circ}_{\gamma_{\rho} (t)};\, t\ge 0) $ is a regular diffusion on $\Gamma $, adapted to $\process{\bF_{\gamma_{\rho}(t)}} $, defined  by~\ref{item_A1}--\ref{item_A2}.
\end{theorem}

\begin{remark}
	\label{rmk_gamma_characterization}
	We observe that 
	\begin{equation}
			\gamma_{\rho}(t) = \int_{0}^{t} \indicB{X_s \not = 0} \vd s,
			\quad \text{for all } t\ge 0. 
	\end{equation} 
	Indeed, if $\mX = (I,X) $ and considering $\mY = (I,Y) := s(\mX)$, 
	from the forthcoming Remark~\ref{rmk_loct_occt}, we have that
	\begin{equation}
		\label{eq_relation_occt_loct}
			\int^{t}_{0} \indicB{Y_s = 0} \vd s = \frac{\rho^{\mY}}{2} \loct{Y}{0+}{t},
			\quad \text{for all } t\ge 0. 
	\end{equation}
	On the other hand, $A_{\rho} $ is strictly increasing and continuous, hence
	$\gamma_{\rho} $ is a time-change and its proper inverse is $A_{\rho} $.
	Thus, from the forthcoming Lemma~\ref{def_TC_local_time} and~\eqref{eq_relation_occt_loct}, 
	\begin{align}
		t = A_{\rho}(\gamma_{\rho}(t)) 
		& = \gamma_{\rho}(t) 
		+ \frac{\rho^{\mY}}{2} \loct{Y^{\circ}}{0}{\gamma_{\rho}(t)}
		= \gamma_{\rho}(t) 
		+ \frac{\rho^{\mY}}{2} \loct{Y}{0}{t}
		\\&= \gamma_{\rho}(t) - \int_{0}^{t} \indicB{Y_s = 0} \vd s
		= \gamma_{\rho}(t) - \int_{0}^{t} \indicB{X_s = 0} \vd s.
	\end{align}
	This yields the assertion. 
\end{remark}

%% New Section
\section{Some preliminary results}
\label{sec_char}

We now present some properties of diffusions on $\Gamma $. 
In particular, (i) we prove that the quadratic variation of the distance to origin process
is a.s. strictly increasing, (ii) we give probabilistic interpretations for the parameters $\rho $ and $(\beta_e;\, e\in E) $ in the gluing condition~\eqref{eq_def_lateral}, (iii) we prove a construction of (sticky) Walsh Brownian motion from
(sticky) reflected Brownian motion and an edge-transition condition,
(iv) we show that a diffusion on $\Gamma $, under Assumption~\ref{assumption_regularity}, 
is a re-scaled NSE diffusion on $\Gamma $.  

For the rest of this section, we will assume $\mX $ is defined on the probability space $\mc P_{\bx} = (\Omega,\process{\bF_t},\Prob_{\bx}) $, such that $\mX_0 = \bx $ holds $\Prob_{\bx} $-almost surely. 
Let us also introduce notations we adopt throughout the paper.

\begin{notation}
	\label{notation_exit_times}
	Let, 
	\begin{itemize}
		\item $T_0 := \inf \{t\ge 0:\; X_t =0 \} $, 
		\item $T_{\bx} := \inf \{t\ge 0:\; \mX_t = \bx\} $,
		for all $\bx \in \Gamma $,
		\item $T_b := \inf \{t\ge 0:\; X_t \not \in [0,b)\} $,
		for all $b> 0 $,
		\item $H_b(\bx) := \Esp_{\bx} \braces{ T_b} $,
		for all $b> 0 $,
		\item $\beta^{\delta}_{e}
		:= \Prob_{\bv} \braces{I(T^{\mX}_{\delta})=e} $,
		for all $e\in E $ and $\delta>0 $.
		\item $l_* := \inf\{l_e;\, e\in E\} $.
	\end{itemize}
\end{notation} 

In case of ambiguity, to indicate that these quantities are related to a process
$\mX $, we will note them $T^{\mX}_{0} $, $T^{\mX}_{\bx} $, $T^{\mX}_{b} $ and 
$H^{\mX}_{b}(\bx) $.

\subsection{Recalls on time changes}

Let us also recall notions and results on time-changed processes from~\cite{RevYor}.
The latter will allow us to manipulate stochastic integrals and local times
of time-changed processes. 

\begin{definition}
	[\cite{RevYor}, Definition~V.1.2]
	\label{def_TC}
	A time change $\gamma := (\gamma(t);\, t\ge 0) $ is a family of stopping times
	such that the mapping $t\mapsto \gamma(t) $ is almost surely increasing. 
\end{definition}

\begin{definition}
	[\cite{RevYor}, Definition~V.1.3]
	\label{def_TC_continuity}
	Let \( \gamma \) be a time-change. 
	The process \( X \) is said to be \( \gamma \)-continuous if
	\[
	X \text{ is constant on each interval } [\gamma(t-), \gamma(t)], \quad t\ge 0.
	\]
\end{definition}

\begin{example}
	Let $A$ be a family of right-continuous $\process{\bF_t} $-adapted processes,
	and $\gamma $ be its right-inverse, i.e.,
	\begin{equation}
		\label{eq_def_right_inverse}
		\begin{aligned}
			\gamma(t) &:= \inf \cubraces{s\ge 0:\; A(s)>t},
			& t&\ge 0.
		\end{aligned}
	\end{equation}
	From \cite[Proposition~V.1.1]{RevYor}, the process $\gamma $ is a time change
	and the process $A $ is a $\process{\bF_{\gamma(t)}} $-time change. 
\end{example}

\begin{example}
	Let $X$ be a semimartingale of quadratic variation $\qv{X} $ and
	$\varphi $ be the right-inverse of $\qv{X} $,  
	in the sense of~\eqref{eq_def_right_inverse}. 
	The process $\varphi $ is a time change. 
	The process $\qv{X} $ is a $\process{\bF_{\varphi(t)}}$-time change. 
\end{example}

We are now ready to state the results presented in the form of Lemmas.

\begin{lemma}
	[\cite{RevYor}, Proposition~V.1.5]
	\label{def_TC_stochastic_integral}
	Let \( \gamma \) be a.s. finite and \( X \) be a continuous \( (\mathcal{F}_t) \)-local martingale.
	
	\begin{enuroman}
		\item If \( X \) is \( \gamma \)-continuous, then \( X \) is a continuous \( (\mathcal{F}_{\gamma(t)}) \)-local martingale and 
		\[
		\qv{ X_{\gamma(\cdot)}} = \qv{ X}_{\gamma(\cdot)}.
		\]
		
		\item If, moreover, \( H \) is \( (\mathcal{F}_t) \)-progressive and 
		$\int_0^t H_s^2 \vd \qv{X}_s < \infty $, a.s. for every $t$,
		then $\int_0^t H_{\gamma(s)}^2 \vd \qv{X}_{\gamma(s)} < \infty $, a.s. for every $t$, and 
		\begin{equation}
			\begin{aligned}
				\int_0^t H_{\gamma(s)}^2 \vd X_{\gamma(s)}
				&= \int_0^{\gamma(t)} H_{s}^2 \, \vd X_{s},
				& \forall\, t&\ge 0. 
			\end{aligned}
		\end{equation}
	\end{enuroman}
\end{lemma}

\begin{lemma}
	[\cite{RevYor}, Exercise~VI.1.27]
	\label{def_TC_local_time}
	Let \( \gamma \) be a process of time changes and \( X \) a \( \gamma \)-continuous semimartingale.
	Then it holds that
	\begin{equation}
		\begin{aligned}
			\loct{X_{\gamma(\cdot)}}{y}{t} &= \loct{X}{y}{\gamma(t)},
			& \forall\,& t\ge 0,
			& \forall\,& y\in \IR.			
		\end{aligned}
	\end{equation}
\end{lemma}

The following result will justify all forthcoming time-change manipulations. 
The proof is deferred to Appendix~\ref{app_strict_monotony}. 

\begin{proposition}
	\label{prop_strict_monotony}
	Let $\mX = (I,X) $ be a regular NSE diffusion on $\Gamma $. 
	The mapping $t \mapsto \qv{X}_t $ is almost surely strictly increasing. 
	Hence, its right-inverse $t \mapsto \varphi(t) $ is also strictly increasing, which 
	qualifies them as proper inverses, i.e., $\varphi(\qv{X}_t) = \qv{X}_{\varphi(t)} = t$, for all $t\ge 0 $. 
\end{proposition}

\subsection{Infinitesimal behavior at $\bv $}

For the next result we may lift Assumption~\ref{assumption_regularity} 
and only assume that $\mX $ is a regular diffusion on $\Gamma $.
The reason is that these result involve only localized versions of the process
by $T_{\delta} $ with $\delta \to 0 $,
which by virtue of Proposition~\ref{prop_FD} are necessarily Feller--Dynkin 
when $\delta $ is small enough.

\begin{proposition}
	\label{prop_node_infinitesimal}
	It holds that
	\begin{enuroman}
		\item $ \lim_{\delta \to 0} \delta^{-1} H_{\delta}(\bv)
		= \rho$, \label{item_node_sticky}
		\item $\lim_{\delta\to 0} \beta^{\delta}_{e} = \beta_{e}$,
		for all $e\in E $. \label{item_node_skew}
	\end{enuroman}
\end{proposition}

\begin{proof}
	[Proof of~\ref{item_node_sticky}]
	We recall Dynkin's formula (see~\cite[Lemma~17.21]{Kal}),
	\begin{equation}
		\label{eq_proof_Dynkin}
		\begin{aligned}
			\Esp_{\bx} \braces{ f(\mX_{T^{}_{\delta}}) - f(\bx) }
			&=
			\Esp_{\bx} \sqbraces{ \int_{0}^{T^{}_{\delta}} \Lop f(\mX_s) \vd s },
			& \forall\, f\in & \dom_{C_0}(\Lop).
		\end{aligned}
	\end{equation}
	Let $B(\bv;\delta) $ be the open ball of $\Gamma$ of radii $\delta>0 $,
	defined as $B(\bv;\delta) := \{ (e,x)\in \Gamma:\; x < \delta,\, e\in E\}$. 
	One can always find $f $ such that 
	$f \in \dom_{C_0}(\Lop) $ such that $\frac 12 \D_{\m_e}\D_{\s_e} f = -1 $
	on $B(0;\delta) $ and $f=0 $ on $B^{c}(0;2\delta) $. 
	The fact that there exists such function $g $ defined locally on the closure of $B(0,\delta) $ is part of Proposition~\ref{prop_Dirichlet_solution}. 
	It thus remains to show that such function $g$ can be extended to a function $f$ of $\dom_{C_0}(\Lop)  $. 
	
	Let $f$ be the function defined as
	\begin{equation}
		f(e,x) := 
		\begin{cases}
			g(e,x), &x<\delta,\\
			h(e,x), &x\in [\delta,2\delta),\\
			0, &x\ge 2\delta,
		\end{cases}
	\end{equation}
	with
	\begin{equation}
		h(e,x) := \int_{\delta }^{x } \int_{\delta }^{y } \frac{1}{\delta} \braces{\zeta - 2\delta}
		\, \m_{e}(\rd \zeta)\, \s_e(\rd y)
		+ A_e \int_{\delta }^{x } \s_e(\rd y)
		+ g(e,\delta).
	\end{equation}
	The constants $(A_e;\, e\in E) $ are chosen such that
	$h(e, 2\delta) = 0 $, for all $e\in E $.
	We also observe that
	\begin{equation}
		h(e,\delta)=g(e,\delta),
		\quad \Lop_e h(e,\delta+) = \Lop_e g(e,\delta-) = -1, 
		\quad \text{and} \quad \Lop_e h(e, 2\delta-) = 0,
	\end{equation}
	for all $e\in E $.
	Therefore, $f\in \dom_{C_0}(\Lop)$.
	
	For such $f$, from~\eqref{eq_proof_Dynkin}, we have that $f(\bx) = H_{\delta}(\bx) = \Esp_{\bx} \sqbraces{T^{}_{\delta}} $, for all
	$\bx\in B(\bv;\delta) $. 
	Also $f $ satisfies the obvious boundary conditions
	$f(\bx) = 0 $, for all $\bx \in \partial B(0;\delta) $.
	The problem hence reads
	\begin{equation}
		\label{eq_proof_stoppingtime_Dirichlet}
		\begin{cases}
			\frac{1}{2}\D_{\m_{e}} \D_{\s_e} u(e,x) = -1, &\forall\, x < 1,\quad \forall\, e\in E,\\
			\rho \Lop_e u(e,0) = \sum_{e'\in E} \beta_{e'} u_{e}'(0),
			&\forall\, e\in E,\\
			u(e,0)=u(e',0), &\forall\, e,e'\in E,\\
			u(e,\delta) = 0, &\forall\, e\in E. 
		\end{cases}
	\end{equation}
	
	From Proposition~\ref{prop_Dirichlet_solution}, the function $H_{\delta} $ is the unique solution to~\eqref{eq_proof_stoppingtime_Dirichlet}.
	From Proposition~\ref{prop_sticky_estimate}, the limit holds.
	This completes the proof.
\end{proof}

\begin{proof}
	[Proof of~\ref{item_node_skew}]
	Let $f\in \dom_{C_0}(\Lop) $.
	We have that
	\begin{equation}
		\label{eq_proof_biasparameters}
		\Esp_{\bv} \braces{ f(\mX_{T^{}_{\delta}}) - f(\bv) }
		= \sum_{e\in E} \beta^{\delta}_e \braces{f_e(\delta) - f_e(0)}.
	\end{equation}
	From Dynkin's formula (see~\cite[Lemma~17.21]{Kal}), if $f\in \dom_{C_0}(\Lop) $,
	\begin{equation}
		\label{eq_proof_vertex_Dynkin}
		\Esp_{\bv} \braces{ f(\mX_{T^{}_{\delta}}) - f(\bv) }
		=
		\Esp_{\bv} \sqbraces{ \int_{0}^{T^{}_{\delta}} \Lop f(\mX_s) \vd s }. 
	\end{equation}
	Regarding the right hand term of~\eqref{eq_proof_vertex_Dynkin}, 
	from~\ref{item_node_sticky}, the Dynkin operator
	(see~\cite[Theorem~17.23]{Kal}), and the definition of $\dom_{C_0}(\Lop) $,
	we have that
	\begin{align}
		\lim_{\delta \to 0} \frac{1}{\delta} \Esp_{\bv} \sqbraces{ \int_{0}^{T^{}_{\delta}} \Lop f(\mX_s) \vd s }
		&= \lim_{\delta \to 0} \braces{ \frac{\Esp_{\bv} \sqbraces{T^{}_{\delta}}}{\delta}
			\frac{1}{\Esp_{\bv} \sqbraces{T^{}_{\delta}}}
			\Esp \sqbraces{ \int_{0}^{T^{}_{\delta}} \Lop f(\mX_s) \vd s }}
		\\
		&= \rho \Lop f(0) = \sum_{e\in E} \beta_e f'_e(0).
	\end{align} 
	Regarding the left hand term of~\eqref{eq_proof_vertex_Dynkin}, 
	\begin{align}
		\lim_{\delta \to 0} \frac{1}{\delta} 	\Esp_{\bv} \braces{ f(\mX_{T^{}_{\delta}}) - f(\bv) }
		&= \lim_{\delta \to 0} \frac{1}{\delta} \sum_{e \in E} \Prob_{\bv} \braces{I(T^{}_{\delta})=e} \braces{f_{e}(\delta)-f_e(0)}.
	\end{align}
	Combining both relations yields that
	\begin{equation}
		\label{eq_proof_node_sumrelation}
		\sum_{e\in E} \beta_e f'_e(0)
		=  \lim_{\delta \to 0} \sum_{e \in E} \beta^{\delta}_e f'_e(0).
	\end{equation}
	
	Since the function \( f \) above is arbitrary in \(\dom_{C_0}(\Lop)\), we may construct, by extension (as in the proof of~\ref{item_node_sticky}), a function \( f \) for some \(\varepsilon > 0\) and $\wh e\in E $ such that
	\begin{equation}
		f(e,x)
		= 
		\begin{dcases}
			2 \int_{0}^{x} \m_e((0,y])\, \s_e(\rd y) +  \frac{\rho}{\beta_{\wh e}} x, 
			& \forall\, x \in [0,\varepsilon),\quad \forall\,  e = \wh e,
			\\
			2 \int_{0}^{x} \m_e((0,y])\, \s_e(\rd y), 
			& \forall\, x \in [0,\varepsilon),\quad  \forall\, e \neq \wh e. 
		\end{dcases}
	\end{equation}
	For such $f $, from~\eqref{eq_proof_node_sumrelation}, we get that~\ref{item_node_skew} holds for $e=\wh e $.
	This completes the proof. 
\end{proof}

\begin{corollary}
	\label{cor_edge_transition_NSE}
	If $\mX $ is NSE (or $\s_e = \s_{e'}$ for all $e,e'\in E $), then 
	\begin{equation}
		\begin{aligned}
			\beta^{b}_{e} &= \beta_e,
			& \forall\, &e\in E,
			& \forall\, &0<b<l_{*}.
		\end{aligned}
	\end{equation}
\end{corollary}

\begin{proof}
	For all  $\delta \in (0,b) $, from Bayes' rule, 
	\begin{align}
		\Prob_{\bv} \braces{ J(T_b) = e }
		={} & \sum_{e'\in E} \Prob_{\bv} \braces{ J(T_b) = e ;\; J(T_\delta) = e'}
		\\ ={}& \sum_{e'\in E} \Prob_{\bv} \braces{ J(T_b) = e |\; J(T_\delta) = e'}
		\Prob_{\bv } \braces{J(T_\delta) = e'}
		\\ ={}& \sum_{e'\in E} \Prob_{(e',\delta)} \braces{ J(T_b) = e }
		\Prob_{\bv } \braces{J(T_\delta) = e'}
		\\ ={}& \Prob_{(e,\delta)} \braces{ J(T_b) = e ;\, T_0 > T_b }
		\Prob_{\bv } \braces{J(T_\delta) = e'}
		\\ &+ \sum_{e'\in E} \Prob_{(e',\delta)} \braces{ J(T_b) = e  ;\, T_0 < T_b}
		\Prob_{\bv } \braces{J(T_\delta) = e'}
		\\ ={}& \frac{\delta}{b} \Prob_{\bv } \braces{J(T_\delta) = e}
		+ \sum_{e'\in E} \frac{b-\delta}{b} \Prob_{\bv} \braces{ J(T_b) = e } \Prob_{\bv } \braces{J(T_\delta) = e'}
		\\ ={}& \frac{\delta}{b} \Prob_{\bv } \braces{J(T_\delta) = e}
		+ \frac{b-\delta}{b} \Prob_{\bv} \braces{ J(T_b) = e } 
	\end{align}
	Hence, 
	\begin{equation}
		\begin{aligned}
			\Prob_{\bv} \braces{ J(T_b) = e } &= \Prob_{\bv } \braces{J(T_\delta) = e},
			& \forall\, 0&< \delta < b.
		\end{aligned}
	\end{equation}
	Taking the limit as $\delta \to 0 $ completes the proof. 
\end{proof}

\subsection{Construction of (sticky) Walsh Brownian motion}

\begin{proposition}
	\label{prop_construction_Walsh}
	Let $\mW^{*} := (J,R^{*}) $ be a diffusion on $\Gamma^{\ext} $ such that
	\begin{enumerate}
		\item $R^{*} $ is a sticky-reflected Brownian motion of stickiness parameter $\rho\ge 0 $,
		\item for every edge $e\in E $, the following limit holds:
		\begin{equation}
			\begin{aligned}
				\lim_{\delta \to 0} \Prob_{\bv} \braces{J(T^{\mW^{*}}_{\delta}) = e} &= \beta_e. 
			\end{aligned}
		\end{equation}
	\end{enumerate}
	Then, $\mW^{*} $ is a sticky Walsh Brownian motion on $\Gamma^{\ext} $ of stickiness parameter $\rho $ and bias parameters $(\beta_e;\, e\in E) $. 		
\end{proposition}

\begin{proof}
	Since $R^{*}$ is a sticky reflected Brownian motion, 
	\begin{equation}
		\begin{aligned}
			\m_{e}(\rd x) &= \vd x + \rho\, \delta_0(\rd x),
			& \s_e(x) &= x,
			& \forall\,& x>0,
			& \forall\,& e\in E. 
		\end{aligned}
	\end{equation}
	From Proposition~\ref{prop_FD}, $\mW $ is Feller--Dynkin.
	We observe that $\Lop_e = \frac 12 \D_{\m_e} \D_{\s_e} = f''(x) $, for all $x> 0 $ and $f\in \dom_{C_0}(\Lop_e) $, where 
	\begin{equation}
		\dom_{C_0}(\Lop_e) \subset C^{2}_0([0,\infty)).
	\end{equation} 
	The law of the reflecting Brownian motion is the same as the absolute value of the (two-sided) sticky Brownian motion (see~\cite{EngPes}).
	Hence, from the Green formula (see~\cite[Corollary~VII.3.8]{RevYor}) on the sticky Brownian motion,  
	\begin{align}
		\frac{1}{\delta} \Esp_{\bv}\sqbraces{T^{\mW^{*}}_{\delta}}
		&= \frac{1}{\delta} \Esp_{0}\sqbraces{T^{R^{*}}_{\delta}}
		= \frac{1}{\delta} \Esp_{0}\sqbraces{T^{B^{*}}_{\delta}}
		\\ &= 
		\frac{1}{\delta}\int_{-\delta}^{\delta} \frac{(0\wedge y + \delta)(\delta - 0 \vee y)}{2\delta}\, \braces{ \rd y + \rho \delta_{0}(\rd y) }
		= \frac{\rho \delta}{2}
		+ \mc O(\delta)
	\end{align}
	From Dynkin's formula~\cite[Lemma~17.21]{Kal},
	\begin{align}
		\frac{1}{\delta} & \braces{\Esp_{\bv}  \sqbraces{ f(\mW^{*}_{T^{\mW}_{\delta}})} - f(\bv)}
		\\ &=
		\frac{\Esp_{\bv}\sqbraces{T^{\mW^{*}}_{\delta}}}{\delta} 
		\braces{ \Lop f(\bv) + \frac{1}{\Esp_{\bv}\sqbraces{T^{\mW^{*}}_{\delta}}}
			\Esp_{\bv} 
			\sqbraces{ \int_{0}^{T^{\mW^{*}}_{\delta}} 
				\Lop f(\mW^{*}_{T^{\mW^{*}}_{\delta}}) - \Lop f(\bv) \vd s}}.
	\end{align}
	Taking to the limit above yields,
	\begin{equation}
		\lim_{\delta \to 0} 
		\frac{1}{\delta} \braces{\Esp_{\bv} \sqbraces{ f(\mW^{*}_{T^{\mW^{*}}_{\delta}})} - f(\bv)}
		= \frac{\rho}{2} \Lop f(\bv).
	\end{equation}
	On the other hand, 
	\begin{align}
		\lim_{\delta \to 0} \frac{1}{\delta} & \braces{\Esp_{\bv}  \sqbraces{f(\mW^{*}_{T^{\mW^{*}}_{\delta}})} - f(\bv)}
		\\ & = \lim_{\delta \to 0} \frac{1}{\delta} \braces{ \sum_{e \in E} \Prob_{\bv} \braces{J(T^{\mW^{*}}_{\delta}) = e} 
			\braces{f_{e}(\delta) - f_{e}(0)}}
		= \sum_{e\in E} \beta_e f'_{e}(0)
	\end{align}
	Combining the last two relations yields the gluing condition~\eqref{eq_def_lateral}. 
	By definition (see Section~\ref{ssec_walsh}), the process $\mW^{*} $ is a sticky Walsh Brownian motion
	on $\Gamma^{\ext} $.
\end{proof}

We observe that the Walsh Brownian motion is a NSE diffusion.
Hence, from Corollary~\ref{cor_edge_transition_NSE} the construction can be simplified as follows.

\begin{corollary}
	Let $\mW^{*} := (J,R^{*}) $ be a diffusion on $\Gamma^{\ext} $ such that:
	\begin{enumerate}
		\item $R^{*} $ is a sticky-reflected Brownian motion of stickiness parameter $\rho \ge 0 $,
		\item for some $b>0 $ and some family $(\beta_e;\, e\in E) $ that satisfies the contraint in~\ref{item_A2} (strict positivity and sum to 1):
		\begin{equation}
			\begin{aligned}
				\Prob_{\bv} \braces{J(T^{\mW^{*}}_{b}) = e} &= \beta_e,
				& \forall\, e&\in E. 
			\end{aligned}
		\end{equation}
	\end{enumerate}
	Then, $\mW^{*} $ is a sticky Walsh Brownian motion on $\Gamma^{\ext} $ of stickiness parameter $\rho $ and bias parameters $(\beta_e;\, e\in E) $. 		
\end{corollary}

\subsection{Change of scale}

We now show how diffusions on $ \Gamma$ behave under a change of scale. For one-dimensional diffusions, this is done by~\cite[Exercise~3.18]{RevYor}. Our motivation is that our time-change characterizations are initially established for NSE diffusions. This requirement is analogous to the standard Dambis--Dubins--Schwarz theorem, where the process must be a local martingale, and to the time-change characterization by It\^o and McKean, where the process must be on natural scale.

The following result provides a way to extend characterizations from NSE diffusions to non-NSE diffusions on $ \Gamma$.

\begin{proposition}
	\label{prop_change_of_scale}
	Let $\mX = (I,X) $ be a regular diffusion on $\Gamma $, defined through~\ref{item_A1}--\ref{item_A2}, that satisfies Assumption~\ref{assumption_regularity}.
	The process $ \mY := s(\mX) = (I(t),\s_{I(t)}(X_t);\;t\ge 0) $ is a NSE regular diffusion on $\Gamma $, of state-space $\Gamma' := \s(\Gamma) $, of speed measures the push-forward measures 
	$ \m^{\mY}_{e} = \m_{e}\circ \q_e$, $e\in E$ and with gluing condition
	\begin{equation}
		\begin{aligned}
			\frac {\rho^{\mY}}{2} \D_{\m^{\mY}_{e}} \D_x f_{e}(0)
			&= \sum_{e' \in E} \beta^{\mY}_{e'} f'_{e}(0), 
			& \text{for all } e&\in E,
		\end{aligned}
	\end{equation}
	where  
	\begin{equation}
		\begin{aligned}
			\rho^{\mY} &= \braces{\sum_{e'\in E} \beta_{e'} \s'_{e}(0) }^{-1}  \rho,
			& \text{and}&
			& \beta^{\mY}_{e} &=  \braces{\sum_{e'\in E} \beta_{e'} \s'_{e}(0) }^{-1} \braces{ \beta_{e'} \q'_{e}(0) }.
		\end{aligned}
	\end{equation}
\end{proposition}

\begin{proof}
	We observe that the process $\mY_t = s(\mX_t) $ is a measurable function of
	$\mX_t $ and hence, satisfies the strong Markov property. 
	Due to the continuity of $s $, the process $\mY $ also has almost surely continuous sample paths.
	Consequently, it is a regular diffusion on $\Gamma' $. It remains only to determine its scale $( \s^{\mY}_e;\, e\in E) $, speed $( \m^{\mY}_e;\, e\in E) $ and gluing condition.
	
	We rely on the local character of the scale and speed
	(see~\cite[Section~7.3]{RevYor}).
	Regarding the scale, from~\cite[Exercise~3.18]{RevYor}, $\mY $ is NSE.
	Regarding the speed, owing to~\eqref{eq_scalespeed} and a change of variables,
	\begin{align}
		\Esp_{x} \sqbraces{T^{\mY}_{a,b}}
		&= \Esp_{\s_e(x)} \sqbraces{T^{\mX}_{\s_e(a),\s_e(b)}}
		\\ &= 2 \int_{\s_e(a)}^{\s_e(b)} G^{(e)}_{\s_e(a),\s_e(b)}(\s_e(x),y) \m_{e}(\rd y)
		\\ &= 2 \int_{\s_e(a)}^{\s_e(b)} G^{(e)}_{\s_e(a),\s_e(b)}(\s_e(x),\q_e(\s_e(y))) \m_{e}(\q_e(\s_e(\rd y)))
		\\ &= 2 \int_{a}^{b} G^{(e)}_{\s_e(a),\s_e(b)}(\s_e(x),\q_e(y)) \m_{e}(\q_e(\rd y)).
	\end{align} 
	for all $e\in E $ and $0\le a<x<b $.
	Hence, from~\cite[VII.3.6,\, VII.3.7]{RevYor}, $\m^{\mY}_e = \m_e \circ \q_e $. 
	
	We may proceed to a partial identification of the domain of the generator.
	The domain of continuity of $\Lop^{\mY}_{e} = \frac 12 \D_{\m^{\mY}_e} \D_{x} $
	is  
	\begin{equation}
		\dom_{C_0}(\Lop^{\mY}_{e})  := \cubraces{ f \in C_0([0,\infty)):\; \Lop^{\mY}_{e} f \in C_0([0,\infty))}. 
	\end{equation}
	We can easily check that 
	\begin{equation}
		\label{eq_proof_domain_relation}
		\begin{aligned}
			f &\in \dom_{C_0}(\Lop_{e})
			& \text{iff}& 
			& f \circ \q_e &\in \dom_{C_0}(\Lop^{\mY}_{e}).
		\end{aligned}
	\end{equation}
	Indeed, let $f \in \dom_{C_0}(\Lop_{e}) $
	and $g =f \circ \q_e  $.
	Then,
	\begin{align}
		\D_{x} g(x)
		&= \lim_{h\to 0} \frac{g(x+h)-g(x)}{h}
		= \lim_{h\to 0} \frac{f(\q_e(x+h))-f(\q_e(x))}{h}
		\\&= \lim_{h\to 0} \frac{f(\q_e(x+h))-f(\q_e(x))}{\s_e(\q_e(x+h))-\s_e(\q_e(x)) }
		= \D_{\s_e} f (\q_e(x)).
	\end{align}
	Also, we have that
	\begin{align}
		\D_{\m_{e}^{\mY}}\D_{x} g(x)
		&
		= \lim_{h\to 0} \frac{\D_{\s_e} f(\q_e(x+h))- \D_{\s_e}f(\q_e(x))}{\m_{e}^{\mY} ([x,x+h))}
		\\ & = \lim_{h\to 0} \frac{\D_{\s_e} f(\q_e(x+h))- \D_{\s_e}f(\q_e(x))}
		{\m_{e} ([\q_e(x),\q_e(x+h)))}
		= \D_{\m_{e}} \D_{\s_e} f(\q_e(x)).
		\label{eq_proof_COS_ms_mx}
	\end{align}
	This proves~\eqref{eq_proof_domain_relation}. 
	
	The gluing condition for $f\in \dom(\Lop)$ is
	\begin{equation}
		\frac{\rho}{2} \D_{\m_e} \D_{\s_e} f_e(0)
		= \sum_{e\in E} \beta_e f_{e}'(0), \quad \text{for all } e\in E.
	\end{equation}
	We observe that
	\begin{equation}
		g'_e(x) = \braces{f_e(\q_e(x))}' = f'_e(\q_e(x)) \q'_e(x) = \frac{f'_e(\q_e(x))}{\s'_e(x)}. 
	\end{equation}
	Hence, from previously established identity~\eqref{eq_proof_COS_ms_mx}, 
	\begin{equation}
		\frac{\rho}{2} \D_{\m_{e}^{\mY}}\D_{x} g_e(x)
		= \sum_{e\in E} \beta_e \s'_e(0) g'_e(0), \quad \text{for all } e\in E.
	\end{equation}
	A renormalization yields
	\begin{align}
		\rho \braces{\sum_{e'\in E} \beta_{e'} \s'_{e}(0) }^{-1} &
		\frac{1}{2} \D_{\m^{\mY}_e} \D_{x} f_{e}(0)
		= 
		\sum_{e'\in E} \sqbraces{ \braces{\sum_{e'\in E} \beta_{e'} \s'_{e}(0) }^{-1} \braces{ \beta_{e'} \s'_{e}(0) }}
		f'_{e'}(0).
	\end{align}
	This completes the proof. 
\end{proof}

%% New Section
\section{Dambis, Dubins--Schwarz for $\mX $}
\label{sec_first_char}

Let $\mX = (I, X)$ be a diffusion on the star graph $\Gamma$, where $I$ denotes the edge occupation process and $X$ represents the distance-to-origin process. In this section, we prove that $\mX$ can be expressed as a time-changed Walsh Brownian motion, where the time-change is the quadratic variation process $\qv{X}$ of $X$. To establish this result, we first show that $X$ can be viewed as a time-changed reflected Brownian motion. We then identify the law of the process $\mW := \mX_{\varphi(t)}$, where $\varphi$ denotes the right inverse of the quadratic variation process $\qv{X} $.

This result serves as a fundamental step in the proof of Theorem~\ref{thm_main}, which
identifies $\qv{X} $ in terms of sample paths. 

\begin{theorem}
	\label{thm_DDS_Walsh}
	Let $\mX=(I,X) $ be the NSE regular diffusion on $\Gamma $, defined by~~\ref{item_A1}--\ref{item_A2}, on the probability space $\mc P_{\bx} := (\Omega,\process{\bF_t},\Prob_{\bx})$, where $\Prob_{\bx}$-a.s.,
	$\mX_0 = \bx $. 
	We consider the time-change $(\varphi(t);\, t\ge 0) $, defined as the right-inverse of 
	$\qv{X} $.
	There exists a Walsh Brownian motion $\mW=(J,R) $ on $\Gamma^{\ext} $ of bias parameter $(\beta_e;\, e\in E) $, adapted to $\process{\bF_{\varphi(t)}}$, defined on an extension
	of $\mc P_x $ such that 
	\begin{equation}
		\label{eq_thm_graph_DDS_Walsh}
		\begin{aligned}
			\mX_t &= \mW_{\qv{X}_t},
			& \forall\,& t\ge 0.
		\end{aligned}
	\end{equation}
\end{theorem}

We now prove a Dambis--Dubins--Schwarz theorem for the distance-to-origin process $X$, where the base process is a reflected Brownian motion rather than a standard Brownian motion. 
The proof is inspired the proof of \cite[Theorem~IV.47.1]{RogWilV2} in the presence of reflecting boundaries. 

\begin{lemma}
	\label{thm_graph_DDS}
	We consider the setting of Theorem~\ref{thm_graph_DDS}. 
	\begin{enuroman}
		\item The process $X$ is a semimartingale. 
		\label{item_semimartingale}
		\item There exists a reflected Brownian motion $R $, adapted to $\process{\bF_{\varphi(t)}}$, defined 
		on an extension of $\mc P_{\bx} $ such that
		$X_t = R_{\qv{X}_t} $, for all $t\ge 0 $.
		\label{item_DDS}
	\end{enuroman}
\end{lemma}

\begin{proof}
	We recall the definitions of $T_0$, $T_b$ and $H_b(\bx) $,
	from Notation~\ref{notation_exit_times}.
	
	We first show that:
	\begin{enumerate}
		\item \label{item_proof_TC_1} The stopped process $X^{T_b} $ is a local submartingale and
		hence a semimartingale; 
		\item \label{item_proof_TC_2} The process $Z = X - X_0 - \frac{1}{2}\loct{X}{0}{} $ is a continuous local martingale,
		where $\loct{X}{0}{} $ is the local time of $X$ at $0$.
	\end{enumerate} 
	
	Regarding~\ref{item_proof_TC_1}, define
	\begin{equation}
		\begin{aligned}
			\rho_t &:= \inf \{u > t : X^{T_b}_u = 0\},
			& \forall\, t&\ge 0.
		\end{aligned}
	\end{equation}
	From~\cite[Corollary~V.46.15]{RogWilV2}, the process $(X_{t \wedge T_b \wedge \rho_s}; t \ge s)$ is a martingale with respect to the filtration $(\mathcal{F}_t)_{t \ge 0}$. This implies that $(X_{t \wedge \rho_s}; t \ge s)$ is a local martingale, reduced by the sequence $(T_k)_{k \in \mathbb{N}}$. 
	
	Since $X$ is non-negative, we have
	\begin{equation}
		\label{eq_submartingale}
		X^{T_b}_{s} = \mathbb{E} \big[ X^{T_b}_{t \wedge \rho_s} \big| \mathcal{F}_{s} \big] \le \mathbb{E} \big[ X^{T_b}_{t} \big| \mathcal{F}_{s} \big].
	\end{equation}
	This shows that $X^{T_b}$ is a submartingale, and consequently, $X$ is a local submartingale. 
	The writing~\eqref{eq_submartingale} yields a Doob-Meyer decomposition for $X$ (see, e.g.,~\cite[Theorem~VI.32.3]{RogWilV2}). 
	Therefore, by definition, $X$ is a semimartingale, which establishes~\ref{item_semimartingale}.
	
	Regarding~\ref{item_proof_TC_2}, since $X$ is a continuous semimartingale, it has a well defined local time field $(\loct{X}{y}{t};\;t \ge 0,\; y\in \IR) $.
	From the Tanaka formula (see, e.g.,~\cite[Theorem~VI.1.2]{RevYor}),
	\begin{equation}
		\label{eq_proof_repr_refl_1}
		\begin{aligned}
			& &X_t &= X^{+}_t = X_0 + Z_t
			+ \frac{1}{2} \loct{X}{0}{t}, \quad \forall\, t \ge 0,
			\\ \text{where}&
			& Z_t&:= \int_{0}^{t} \indicB{X_s>0} \vd X_s, \quad \forall\, t \ge 0. 
		\end{aligned}
	\end{equation}
	The process $Z$ is a local martingale.
	To prove this, we consider the
	following sequences of stopping times.
	For all  $\delta >0$, let $(\sigma^{\delta}_k,\tau^{\delta}_k;\; k\in \IN_0) $
	defined recursively by $\sigma^{\delta}_0 := 0$, and
	\begin{equation}
		\label{eq_proof_tausigma_stopping_times_def}
		\begin{aligned}
			\tau^{\delta}_k &:= \inf\{t>\sigma^{\delta}_{k}:\; X_t=0\},
			& \forall\, k &\in \IN_0,\\ 
			\sigma^{\delta}_k &:= \inf\{t>\tau^{\delta}_{k-1}:\; X_t=\delta\},
			& \forall\, k &\in \IN.
		\end{aligned}
	\end{equation} 
	We observe that each random interval $[\sigma^{\delta}_k,\tau^{\delta}_k) $
	is subset of an excursion time interval of $X$.
	For all $k\in \IN_0 $ and $\delta>0 $, we consider the process
	\begin{equation}
		\begin{aligned}
			N^{k,\delta}_t &:= \int_{0}^{t} \indicB{t\in (\sigma^{\delta}_k,
				\tau^{\delta}_k]} \vd X_s
			= X_{\tau^{\delta}_k\wedge t} - X_{\sigma^{\delta}_k\wedge t},
			& \forall\, &t \ge 0,
		\end{aligned}
	\end{equation}
	which is a local martingale, reduced by $(\sigma^{\delta}_k + T_n \circ \theta_{\sigma^{\delta}_k} ;\;n\in \IN) $. 
	Its quadratic variation is 
	\begin{equation}
		\begin{aligned}
			\qv{N^{k,\delta}}_t
			&= \int_{0}^{t} \indicB{t\in (\sigma^{\delta}_k,
				\tau^{\delta}_k]} \vd \qv{X}_s,
			& t& \ge 0.
		\end{aligned}
	\end{equation} 
	The process $N^{\delta} := \sum_{k} N^{k,\delta} $ is also a local martingale
	with quadratic variation
	\begin{equation}
		\begin{aligned}
			\qv{N^{\delta}}_t
			&= \sum_{k\in \IN_0} \int_{0}^{t} \indicB{t\in (\sigma^{\delta}_k,
				\tau^{\delta}_k]} \vd \qv{X}_s,
			& t& \ge 0.
		\end{aligned}
	\end{equation} 
	From the monotone convergence theorem, it holds almost surely that
	\begin{equation}
		\begin{aligned}
			\qv{N^{\delta}}_t &\xlra{\delta \to 0}{} \int_{0}^{t} \indicB{X_s>0} \vd \qv{X}_s = \qv{Z}_t 
		\end{aligned}
	\end{equation}
	Hence, from, e.g.,~\cite[Theorem~IV.4.1]{RevYor}, $N^{\delta} $ converges locally uniformly in time, in $L^{2} $, 
	as $\delta \to 0 $ to $Z$.
	
	Since $X = y + Z + \frac{1}{2} \loct{X}{0}{} $ and $\loct{X}{0}{} $ increases only when $X=0 $, from Skorokhod's lemma (see, e.g., \cite[Lemma~V.6.2]{RogWilV2}),
	\begin{equation}
		\frac{1}{2}\loct{X}{0}{t}
		= 0 \vee \braces{ \sup \{ -y -Z_s ;\; s\le t \} }
	\end{equation} 
	
	Let $(\psi(t);\;t\ge 0) $ be the right-inverse of $\qv{Z} $.
	From \cite[Theorem~IV.34.11]{RogWilV2}, there exists a $\process{\bF_{\psi(t)}}$-adapted standard Brownian motion $W$, defined on an extension of $\mc P_x $, such that 
	\begin{equation}
		\label{eq_proof_repr_refl_3}
		\begin{aligned}
			Z_t &= W_{\qv{Z}_t},
			& \forall\, t\ge 0.
		\end{aligned}
	\end{equation}
	From L\'evy's presentation of the reflecting Brownian motion
	(see, e.g., \cite[\S V.6.2]{RogWilV2}), the process
	$R $ defined as
	\begin{equation}
		\label{eq_proof_repr_refl_2}
		\begin{aligned}
			R_t &:= y + W_t + \frac{1}{2} \loct{W}{0}{t},
			&\forall\, t&\ge 0,
		\end{aligned}
	\end{equation}
	is an $\process{\bF_{\psi(t)}}$-adapted reflecting Brownian motion.
	
	From Lemma~\ref{def_TC_local_time}, $\loct{W}{0}{\qv{Z}_t}
	= \loct{Z}{0}{t} $, for all $t\ge 0 $.
	Thus, from~\eqref{eq_proof_repr_refl_1}, \eqref{eq_proof_repr_refl_3} and~\eqref{eq_proof_repr_refl_2}, $\qv{X} = \qv{Z} $, $\psi = \varphi $, and $X_t = R_{\qv{X}_t} $, for all $t\ge 0 $.
	This proves~\ref{item_DDS}, which completes the proof. 
\end{proof}

\begin{proof}
	[Proof of Theorem~\ref{thm_DDS_Walsh}]
	Let $\mW $ be the process  
	$(R,J) $, where $R$ is issued from Theorem~\ref{thm_graph_DDS}
	and $J $ is defined by $J(\qv{X}_t)=I(t) $, $t\ge 0 $.
	We first observe that $\qv{X} $ is an additive functional of $\mX $ that 
	satisfies the hypotheses of~\cite[Theorem~10.10]{Dyn1}. %p.320 
	Hence, the process $\mW $ is Markov with respect to 
	the filtration $\process{\bF_{\varphi(t)}}$.
	Its scale and speed $(\s_e,\m_e;\, e\in E) $ are
	\begin{equation}
		\begin{aligned}
			\s^{\mW}_e &= x,
			& \m^{\mW}_e &= \vd x,
			& x&>0. 
		\end{aligned}
	\end{equation}
	The process satisfies Assumption~\ref{assumption_regularity}, hence
	from Proposition~\ref{prop_FD}, it is a Feller and strong Markov.

	Let $T^{\mX}_b $, $T^{\mW}_b $, $b>0 $, be defined as
	\begin{align}
		T^{\mX}_b :={} \inf\{t\ge 0: \; X_t = b\} \quad \text{and} \quad
		T^{\mW}_b :={} \inf\{t\ge 0: \; R_t = b\} .
	\end{align} 
	We observe that $T^{\mX}_b = \qv{X}^{-1}_{T^{\mW}_b} $ and that
	\begin{align}
		\Prob_{\bv} \braces{ J(T^{\mW}_{\delta}) = e }
		= \Prob_{\bv} \braces{ I(\varphi(T^{\mW}_{\delta})) = e }
		= \Prob_{\bv} \braces{ I(T^{\mW}_b) = e },
	\end{align}
	for all $e\in E $ and $b>0 $.
	Taking the limit yields that
	\begin{equation}
		\lim_{\delta \to 0} \Prob_{\bv} \braces{ J(T^{\mW}_{\delta}) = e }
		= \beta_e. 
	\end{equation}
	
	From Corollary~\ref{prop_construction_Walsh}, the process $\mW $ is a Walsh Brownian motion on
	$\Gamma^{\ext} $, adapted to $\process{\bF_{\varphi(t)}}$, with bias parameters $(\beta_e;\, e\in E) $.
	This completes the proof.
\end{proof}

We now prove several corollaries of Theorem~\ref{thm_DDS_Walsh}, culminating in our first (non-homogeneous) occupation times formula. 
This formula is later refined in Proposition~\ref{prop_occupation_ItoMcKean}.

\begin{corollary}
	\label{cor_DDS_rescaled}
	Consider the following:
	\begin{enumerate}
		\item Let $\mX =(I,X)$ be a diffusion on $\Gamma $, defined by~\ref{item_A1}--\ref{item_A2}, on the probability space $\mc P_{\bx} := (\Omega,\process{\bF_t},\Prob_{\bx})$, where $\Prob_{\bx}$-a.s.,
		$\mX_0 = \bx $.
		\item Let $\mY = (I',Y) $, with $\mY = s(\mX) $, and $(\varphi(t);\, t\ge 0) $ be the right-inverse of $\qv{Y} $.
		\item Let $q $ be the inverse scale of $\mX $, defined in~\eqref{eq_q}.
	\end{enumerate}
	Under Assumption~\ref{assumption_regularity}, the following results hold: 
	\begin{enuroman}
		\item There exists a Walsh Brownian motion $\mW=(J,R) $ on $\Gamma^{\ext} $ of bias parameters
		$(\beta^{\mY};\, e\in E) $, adapted to $\process{\bF_{\varphi(t)}} $, defined on an extension
		of $\mc P_x $, such that 
		$\mX = q(\mW_{\qv{X_t}}) $, for all $t\ge 0 $. \label{item_DDS_rescaled}
		
		\item The process $X $ is a semimartingale. \label{item_semimartingale_rescaled}
	\end{enuroman} 
	The family $(\beta^{\mY};\, e\in E) $ is given in
	Proposition~\ref{prop_change_of_scale}. 
\end{corollary}

\begin{proof}
	From Proposition~\ref{prop_change_of_scale}, the process $\mY $
	is a NSE diffusion on $s(\Gamma) $ of
	speed measure $(m^{\mY}_e;\, e\in E) $,
	stickiness $\rho^{\mY} $ and bias parameters
	$(\beta^{\mY}_e;\, e\in E) $. 
	Hence, \ref{item_DDS_rescaled} is consequence of  Theorem~\ref{thm_DDS_Walsh} applied to $\mY $.
	
	Regarding the first assertion~\ref{item_semimartingale_rescaled}, from Lemma~\ref{thm_graph_DDS},
	the process $Y $ is a semimartingale as time-changed semimartingale. 
	We may decompose the process $X$ as 
	\begin{equation}
		X_t = \sum_{e\in E} \indicB{I(t)=e} \q_{e}(Y_t).
	\end{equation}
	From the discussion in~\cite[Section~5]{ccinlar1980semimartingales}, all functions $\q_e $ are differences of convex functions.
	Hence, from the It\^o--Tanaka formula (see \cite[Theorem~VI.1.5]{RevYor})
	\begin{equation}
		\label{eq_proof_Ito_qe}
		\begin{aligned}
			\vd \q_{e}(Y_t)
			= \q_{e}'(Y_t) \vd Y_t + \frac{1}{2} \int_{[0,\infty) } \vd \loct{Y}{y}{t}
			\, \q_{e}''(\rd y).
		\end{aligned}
	\end{equation}
	Multiplying~\eqref{eq_proof_Ito_qe} with the predictable process 
	$(\indicB{I(t)=e};\, t\ge 0) $, integrating by time, and summing over $e\in E $ yields
	\begin{equation}
		X_t = X_0 + \sum_{e\in E} \int_{0}^{t} \indicB{I(s)=e} \q'_e(Y_s) \vd Y_s
		+  \sum_{e\in E}  \int_{[0,\infty) } \int_{0}^{t} \indicB{I(s)=e} \vd \loct{Y}{y}{s} \, \q_{e}''(\rd y), 
	\end{equation}
	which is an explicit Doob-Meyer decomposition of $X$.
	This completes the proof. 
\end{proof}

We now present our first occupation times formula that we will
upgrade later in Section~\ref{sec_occupation_II}. 

\begin{corollary}
	[occupation times formula I]
	We consider the assumptions of Corollary~\ref{cor_DDS_rescaled}.
	It holds that 
	\begin{equation}
		\int_{0}^{t} f_{I(s)} (s,X_s)  \vd \qv{X}_s 
		= \sum_{e\in E} \int_{\IR} \int_{0}^{t}  \indicB{I(s)=e} f_{e}(s,y) \vd \loct{X}{y}{s} \vd y,
	\end{equation} 
	for all $t\ge 0 $ and all measurable functions $f $ on $[0,\infty) \times \Gamma $. 
\end{corollary}

\begin{proof}
	Since $X$ is a semimartingale, we can apply \cite[Exercise~VI.1.15]{RevYor} 
	on the measurable function 
	\begin{equation}
		\begin{aligned}
			f(t,x) &:= \sum_{e\in E}\indicB{I(t) = e} f_e(t,x),
			& t&\ge 0,
			& x&\ge 0.
		\end{aligned}
	\end{equation}
	This completes the proof. 
\end{proof}

%% New Section
\section{On the local time $\mX $}
\label{sec_further}

In this section, we define and characterize the local time of a regular (NSE) diffusions \(\mX = (I, X)\) on \(\Gamma\). 
Specifically, we express the local time of \(\mX\) in terms of the local time of its distance-to-origin processes. 
Additionally, we derive an It\^o--Tanaka formula for such diffusions.

To establish these results, we first prove them for Walsh Brownian motion and then extend them to \(\mX\), using the time-change characterization provided in Theorem~\ref{thm_DDS_Walsh}. These results play a key role in the proof of Theorem~\ref{thm_main}.

We first define the local time of a diffusion $\mX=(I,X) $ on $\Gamma $.
Regarding $X$, since it is a semimartingale (Corollary~\ref{cor_DDS_rescaled}), from \cite[Corollary~VI.1.9]{RevYor}, its local time is the random field defined as the almost sure limit
\begin{equation}
	\begin{aligned}
		\loct{X}{y+}{t} &:= \lim_{\delta \to 0} \frac{1}{\delta}
		\int_{0}^{t} \indicB{X_s \in [y, y+\delta)} \vd \qv{X}_s, 
		& t &\ge 0,
		& y &\in \IR. 
	\end{aligned}
\end{equation}
We define the \textit{symmetric local time} $(\loct{\mX}{\bv}{t};\, t \ge 0)$ 
and \textit{the directional local times} $(\loct{\mX}{e,\bv}{t};\, t \ge 0,\, e \in E)$ of $\mW$ at the junction vertex $\bv$ as the random fields defined as the almost sure limits
\begin{align}
	\label{eq_loctime_0_defX}
	\loct{\mX}{\bv}{t} &:= \lim_{\delta \to 0} \frac{1}{\delta}
	\int_{0}^{t} \indicB{X_s \in [0, \delta)} \vd \qv{X}_s,
	\quad \forall\, t \ge 0,
	\\
	\loct{\mX}{e,\bv}{t} &:= \lim_{\delta \to 0} \frac{1}{\delta}
	\int_{0}^{t} \indicB{X_s \in [0, \delta);\, I(s) = e} \vd \qv{X}_s,
	\quad \forall\, t \ge 0, \quad \forall\, e\in E.
	\label{eq_txt_def_loctX}
\end{align}
We also define the local time of $\mX$ at a point $\bx = (e, x) \in \Gamma$, with $x > 0$, as the almost sure limit
\begin{equation}
	\label{eq_txt_def_loct2X}
	\loct{\mW}{\bx}{t} := \lim_{\delta \to 0} \frac{1}{\delta}
	\int_{0}^{t} \indicB{X_s \in [x, x+\delta);\, I(s) = e} \vd \qv{X}_s,
	\quad \forall\, t \ge 0.
\end{equation}

We observe that
\begin{equation}
	\label{eq_rel_loctimes}
	\loct{X}{0+}{t} = 
	\loct{\mX}{\bv}{t} = \sum_{e \in E} \loct{\mX}{e,\bv}{t},
\end{equation}
where $\loct{R}{0+}{t}$ denotes the right-sided local time of $R$ at $0$.

\begin{remark}
	It is worth noting that the directional local times of $\mX $ at $\bv $ plays the role of (right/left) local time for one-dimensional processes.
	Similarly, the symmetric local time of $\mX $ at $\bv $ plays the role of symmetric
	local time for one-dimensional processes.
\end{remark}

In the case of a Walsh Brownian motion $\mW := (J, R)$ on  $\Gamma^{\ext} $, things get simpler.  Indeed, from L\'evy's presentation of the reflected Brownian motion $\qv{R}_t = t $, for $t\ge 0 $ and hence, the local time of $R$ at $y \ge 0$ is defined as the almost sure limit
\begin{equation}
	\begin{aligned}
		\loct{R}{y+}{t} &:= \lim_{\delta \to 0} \frac{1}{\delta}
		\int_{0}^{t} \indicB{R_s \in [y, y+\delta)} \vd s, 
		& t &\ge 0.
	\end{aligned}
\end{equation}
In this case, for all $t\ge 0 $, $e\in E $ and $\bx =(e,x) $, with $x>0 $,
we have
\begin{align}
	\label{eq_loctime_0_def}
	\loct{\mW}{\bv}{t} &:= \lim_{\delta \to 0} \frac{1}{\delta}
	\int_{0}^{t} \indicB{R_s \in [0, \delta)} \vd s,
	\\
	\loct{\mW}{e,\bv}{t} &:= \lim_{\delta \to 0} \frac{1}{\delta}
	\int_{0}^{t} \indicB{R_s \in [0, \delta);\, J(s) = e} \vd s,
	\label{eq_txt_def_loct}
	\\
	\loct{\mW}{\bx}{t} &:= \lim_{\delta \to 0} \frac{1}{\delta}
	\int_{0}^{t} \indicB{R_s \in [x, x+\delta);\, J(s) = e} \vd s.
	\label{eq_txt_def_loct2}
\end{align}

We now prove several additional properties of the local times. First, we express the directional local times of $\mW$ in terms of the symmetric local time. Second, we express the local time of $\mW$ at $\bx = (e, x)$ with $x > 0$ in terms of the local time of $R$ at $x$. Finally, we prove an It\^o--Tanaka formula for $\mW$.

The main tool for our proofs is a technique introduced in~\cite{salminen2024occupation}, which involves isolating a single edge of $\Gamma^{\ext}$ and collapsing all remaining edges into a single edge $\wh e$. We define the process
\begin{equation}
	\label{eq_txt_def_Yprocess}
	Z_t := 
	\begin{cases}
		R_t, & J(t) = e, \\
		-R_t, & J(t) \ne e,
	\end{cases}
	\quad \forall\, t \ge 0.
\end{equation}
The process $Z$ can be shown to be a skew Brownian motion with bias parameter $\beta_e$ (see Corollary~\ref{cor_edge_transition_NSE} and the construction of the skew Brownian motion from~\cite[Problem~1, pp. 115--116]{ItoMcKean96} and~\cite[Example~5.7]{salisbury1986construction}).

\begin{proposition}
	\label{prop_loctime_relationI}
	Let $\mX = (I,X) $ be a diffusion on $\Gamma $. 
	It holds that
	\begin{equation}
		\int_{0}^{t} \indicB{I(s)=e} \, \vd \loct{X}{y+}{t}
		= \loct{\mX}{(e,y)}{t},
	\end{equation}
	for all $e\in E $ and $y>0 $. 
\end{proposition}

\begin{proof}
	Let $\delta < y $ and $(\tau^{e,\delta}_k,\, \sigma^{e,\delta}_k;\; k \ge 0) $
	be the sequences of stopping times defined recursively by $\sigma^{e,\delta}_0 := 0$, and
	\begin{equation}
		\label{eq_proof_tausigma_stopping_times_def2}
		\begin{aligned}
			\tau^{e,\delta}_k &:= \inf\{t>\sigma^{\delta}_{k}:\; X_t=0\},
			& \forall\, k &\in \IN_0,\\ 
			\sigma^{e,\delta}_k &:= \inf\{t>\tau^{\delta}_{k-1}:\; I(t)=e,\; X_t=\delta\},
			& \forall\, k &\in \IN.
		\end{aligned}
	\end{equation} 
	From~\cite[Corollary~VI.1.9]{RevYor} for random and non-random times, 
	\begin{align}
		\int_{0}^{t} \indicB{J(s)=e} \, \vd \loct{X}{y+}{t}
		&= \sum_{k\in \IN_0} \int_{\sigma^{e,\delta}_k}^{\tau^{e,\delta}_k} \, \vd \loct{X}{y+}{t} 
		= \sum_{k\in \IN_0} \braces{\loct{X}{y+}{\tau^{e,\delta}_k} - \loct{X}{y+}{\sigma^{e,\delta}_k}}
		\\ &= \sum_{k\in \IN_0} \lim_{h \to 0}
		\frac{1}{h} \int_{\sigma^{e,\delta}_k}^{\tau^{e,\delta}_k}
		\indicB{X_s \in [y,y+h)} \vd \qv{X}_s 
		\\ &= \lim_{h \to 0}
		\frac{1}{h} \int_{0}^{t}
		\indicB{X_s \in [y,y+h);\; J(s)=e} \vd \qv{X}_s
		= \loct{\mX}{(e,y)}{t}.
	\end{align}
	This completes the proof. 
\end{proof}

\begin{proposition}
	\label{prop_local_times}
	Let $\mX $ be a NSE diffusion defined by~\ref{item_A1}--\ref{item_A2}.
	Then $\loct{\mX}{e,\bv}{t} = \beta_e \loct{X}{0+}{t} $, for all $t\ge 0 $ and $e\in E $.
\end{proposition}

\begin{proof}
	We first prove the result for a Walsh Brownian motion $\mW=(J,R) $ of bias parameters
	$(\beta_e;\, e\in E) $.
	We consider the process $Z$ defined in~\eqref{eq_txt_def_Yprocess}.
	From the path-space characterization  of $Y$, we have that $\qv{Z}_t =t $, for all $t\ge 0 $.  
	Hence, its symmetric local time at $0$ is 
	\begin{equation}
		\loct{Z}{0}{t} := \lim_{\delta \to 0} \frac{1}{\delta}
		\int_{0}^{t} \indicB{|Z_s| < \delta} \vd s, \quad t\ge 0,
	\end{equation}
	and its right local time at $0$ is
	\begin{equation}
		\loct{Z}{0+}{t} := \lim_{\delta \to 0} \frac{1}{\delta}
		\int_{0}^{t} \indicB{0 \le Z_s < \delta} \vd s, \quad t\ge 0.
	\end{equation}
	By definition of $Z$, we observe that $\loct{Z}{0}{t} = \loct{\mW}{\bv}{t}
	= \loct{R}{0+}{t}$ and 
	$\loct{Z}{0+}{t} = \loct{\mW}{e,\bv}{t}$.
	From \cite[Theorem~2.1]{borodin2019local}, it holds that 
	$\loct{Z}{0+}{t} = \beta \loct{Z}{0}{t} $, which leads to  
	\begin{equation}
		\label{eq_proof_dir_symm_WalshBM}
		\begin{aligned}
			\loct{\mW}{e,\bv}{t} & = \beta_e \loct{R}{0+}{t},
			& \forall\, t&\ge 0,
			& \forall\, e&\in E.
		\end{aligned}
	\end{equation}
	
	Let's now assume $\mX $ is NSE. 
	From Theorem~\ref{thm_DDS_Walsh}, there exists a Walsh Brownian motion $\mW=(J,R) $ 
	of bias parameters $(\beta_e;\, e\in E) $, adapted to $\process{\bF_{\varphi(t)}} $ defined on an extension of the probability space such that 
	\begin{equation}
		\begin{aligned}
			\mX_t &= \mW_{\qv{X}_t} 
			& \mW_t &= \mX_{\varphi(t)}
			&\forall\, t&\ge 0.
		\end{aligned}
	\end{equation}
	From Lemma~\ref{def_TC_local_time},
	$\loct{X}{0}{t} = \loct{R}{0}{\qv{X}_t} $, for all $t\ge 0 $.
	For the directional local times, from Lemma~\ref{def_TC_stochastic_integral}, 
	\begin{align}
		\loct{\mX}{\bv,e}{t}
		&= \lim_{\delta \to 0} \frac{1}{\delta}
		\int_{0}^{t} \indicB{X_s \in [0, \delta);\, I(s) = e} \vd \qv{X}_s
		\\ &= \lim_{\delta \to 0} \frac{1}{\delta}
		\int_{0}^{t} \indicB{R_{\qv{X}_s} \in [0, \delta);\, J(\qv{X}_s) = e}  \vd \qv{X}_s
		\\ &= \lim_{\delta \to 0} \frac{1}{\delta}
		\int_{0}^{\qv{X}_t} \indicB{R_{s} \in [0, \delta);\, J(s)=e}  \vd s
		= \loct{\mW}{\bv,e}{\qv{X}_t}. 
	\end{align}
	Combining the above with the relation~\eqref{eq_proof_dir_symm_WalshBM}
	for Walsh Brownian motion, we have that
	\begin{equation}
		\begin{aligned}
			\loct{\mX}{\bv,e}{t} &= \beta_e \loct{X}{0}{t},
			& \forall\,& t\ge 0,
			& \forall\, e&\in E.
		\end{aligned}
	\end{equation}
	This is the desired relation.
\end{proof}

\begin{proposition}[It\^o--Tanaka Formula for $\mW$]
	\label{prop_ItoTanaka2}
	Let $\mW = (J, R)$ be a Walsh Brownian motion of bias parameters $(\beta_e;\, e \in E)$. 
	For any function $f := \bigoplus_{e \in E} f_e$ on $\Gamma^{\ext}$, where $(f_e;\, e \in E)$ is a family of differences of convex functions, the following holds for all $t \ge 0$:
	\begin{align}
		f(J(t), R_t)
		&= f(J(0), R_0)
		+ \sum_{e \in E} \int_{0}^{t} \indicB{J(s) = e} (f_e)'_- (R_s) \vd R_s
		\\
		&\quad + \frac{1}{2} \sum_{e \in E} \int_{(0, \infty)} \int_{0}^{t} \indicB{J(s) = e} \vd \loct{R}{y+}{s} \, f_{e}''(\rd y)
		+ \frac{1}{2} \sum_{e \in E} \beta_e \loct{R}{0+}{t} \, f_{e}''(\{0\}),
	\end{align}
	where $(\loct{R}{y+}{t};\, t \ge 0, y \ge 0)$ is the local time field of $R$.
\end{proposition}

\begin{proof}
	For each $e\in E $, let $g^{(e)} $ be the function defined on $\Gamma^{\ext} $
	as
	\begin{equation}
		g^{(e)}(e',x) :=
		\begin{cases}
			f(e,x), & e=e',\; x\ge 0,\\
			f(e,0), & \text{otherwise}.
		\end{cases}
	\end{equation}
	Let also $\wt f_e $ be the extension of $f_e $ on $\IR $, defined as
	\begin{equation}
		\wt f_e(x) := 
		\begin{cases}
			f_e(x), & x\ge 0,  \\
			f_e(0), & x< 0.
		\end{cases}
	\end{equation}
	Hence, for these functions we have that $g^{(e)}(\mW_t) = \wt f_e(Z_t) $, for all $t\ge 0 $. 
	From the It\^o--Tanaka formula (\cite[Theorem~VI.1.5]{RevYor}), we have that 
	\begin{align}
		\vd \wt f_e(Z_t) 
		&=  (\wt f_e)'_- (Z_t) \vd Z_t
		+ \frac{1}{2} \int_{\IR} \vd \loct{Z}{y+}{t} \,\wt f''_e (\rd y)
		\\
		&= (\wt f_e)'_- (Z_t) \vd Z_t
		+ \frac{1}{2} \int_{(0,\infty)} \vd \loct{Z}{y+}{t} \,\wt f''_e (\rd y)
		+ \frac{1}{2} \loct{Z}{0+}{t} f''_e (\{0\}).
	\end{align}
	Observing that $\loct{Z}{y+}{t} = \loct{\mW}{e,y+}{t} $, and from Proposition~\ref{prop_loctime_relationI},
	\begin{equation}
		\vd \wt f_e(Z_t) = (\wt f_e)'_- (Z_t) \vd Z_t
		+ \frac{1}{2} \int_{(0,\infty)} \vd \loct{\mW}{e,y+}{t} \,\wt f''_e (\rd y)
		+ \frac{1}{2} \beta_e \loct{R}{0+}{t} f''_e (\{0\}).
	\end{equation}
	Multiplying by $\indicB{J(t)=e} $ and integrating in time yields that
	\begin{align}
		g^{(e)}(J(t),R_t)
		={}& g^{(e)}(J(0),R_0)
		+ \int_{0}^{t} \indicB{J(s)=e} (g^{(e)})'_- (e,R_s) \vd R_s
		\\ &+ \frac{1}{2} \int_{(0,\infty)} \indicB{J(s)=e} \vd \loct{\mW}{e,y}{t} \,(g^{(e)}_e)'' (e,\rd y) 
		\\
		&+ \frac{1}{2} \beta_e \loct{R}{0+}{t} (g^{(e)}_e)'' (\{0\}).
	\end{align}
	Summing over all edges, 
	\begin{align}
		\sum_{e\in E} g^{(e)}(J(t),R_t)
		={}& \sum_{e\in E} g^{(e)}(J(0),R_0)
		+ \sum_{e\in E} \int_{0}^{t} \indicB{J(s)=e} (g^{(e)})'_- (e,R_s) \vd R_s
		\\ &+ \frac{1}{2} \sum_{e\in E} \int_{(0,\infty)} \int_{0}^{t}  \indicB{J(s)=e} \vd \loct{\mW}{e,y}{t} \,(g^{(e)}_e)'' (e,\rd y) 
		\\
		&+ \frac{1}{2} \sum_{e\in E} \beta_e \loct{R}{0+}{t} (g^{(e)}_e)'' (\{0\}).
	\end{align}
	From the definition of $g^{(e)} $,
	\begin{align}
		\sum_{e\in E} g^{(e)}(J(t),R_t) &= f(J(t),R_t) + \braces{\card E - 1} f(e,0),
		\\
		\sum_{e\in E} g^{(e)}(J(0),R_0) &= f(J(0),R_0) + \braces{\card E - 1} f(e,0). 
	\end{align}
	Consequently, substracting the adequate quantity yields 
	\begin{align}
		f(J(t),R_t)
		={}& f(J(0),R_0)
		+ \sum_{e\in E} \int_{0}^{t} \indicB{J(s)=e} (g^{(e)})'_- (e,R_s) \vd R_s
		\\ &+ \frac{1}{2} \sum_{e \in E} \int_{(0,\infty)} \int_{0}^{t}  \indicB{J(s)=e} \vd \loct{\mW}{e,y}{s} \, (g^{(e)})'' (e,\rd y)
		\\
		&+ \frac{1}{2} \sum_{e\in E} \beta_e \loct{R}{0+}{t} (g^{(e)}_e)'' (\{0\}).
	\end{align}
	From Propositions~\ref{prop_loctime_relationI},~\eqref{eq_rel_loctimes} and the definition of $g^{(e)} $, the proof is finished. 
\end{proof}

\begin{proposition}[It\^o--Tanaka Formula for $\mX$ NSE]
	\label{prop_ItoTanaka_NSE}
	Let $\mX := (I, X)$ be a NSE diffusion on $\Gamma$ defined via~\ref{item_A1}--\ref{item_A2}. 
	For any function $f := \bigoplus_{e \in E} f_e$ on $\Gamma$, where $(f_e;\, e \in E)$ is a family of differences of convex functions, the following holds for all $t \ge 0$:
	\begin{multline}
		f(I(t), X_t)
		= f(I(0), X_0)
		+ \sum_{e \in E} \int_{0}^{t} \indicB{I(s) = e} (f_e)'_- (X_s) \vd X_s
		\\
		 + \frac{1}{2} \sum_{e \in E} \int_{(0, \infty)} \int_{0}^{t} \indicB{I(s) = e} \vd \loct{X}{y+}{t} \, f_e''(\rd y)
		+ \frac{1}{2} \sum_{e \in E} \beta_e \loct{X}{0+}{t} \, f_e''(\{0\}),
		\label{eq_prop_general_ItoTanaka}
	\end{multline}
	where $(\loct{X}{y+}{t};\, t \ge 0, y \ge 0)$ is the local time field of $X$.
\end{proposition}

\begin{proof}
	Assume $\mX $ is defined on the probability space $\mc P_{\bx} := (\Omega,\process{\bF_t},\Prob_{\bx}) $.
	From Theorem~\ref{thm_DDS_Walsh}, there exists a $(\beta_e;\, e\in E) $-Walsh Brownian motion $\mW := (J,R) $, adapted to $\process{\bF_{\phi(t)}} $, defined on an extension of $\mc P_{\bx} $, such that 
	$\mX_t = \mW_{\qv{X}_t}$, for all $t\ge 0 $.
	The time change $\phi $ is the right-inverse of $\qv{X} $.
	
	From Proposition~\ref{prop_ItoTanaka2},
	\begin{align}
		f(J(\qv{X}_t),R_{\qv{X}_t})
		={}& f(I(0),X_0)
		+ \sum_{e\in E} \int_{0}^{\qv{X}_t} \indicB{J(s)=e} (f_e)'_- (R_s) \vd R_s
		\\ &+ \frac{1}{2} \sum_{e \in E} \int_{(0,\infty)} \int_{0}^{\qv{X}_t} \indicB{J(s)=e} \vd \loct{R}{y+}{t} \,f_e'' (\rd y)
		\\ &+ \frac{1}{2} \sum_{e \in E} 
		\beta_e \loct{R}{0+}{\qv{X}_t} \,f_e'' (\{0\}), \quad \forall\, t\ge 0. 
	\end{align} 
	From Lemmas~\ref{def_TC_stochastic_integral} and~\ref{def_TC_local_time}, we get~\eqref{eq_prop_general_ItoTanaka}.
	This completes the proof. 
\end{proof}

%% New Section
\section{Proofs of main results}
\label{sec_main_proof}

We now prove the main results of the paper (Theorems~\ref{thm_main} and \ref{thm_main2}). 
In Theorem~\ref{thm_DDS_Walsh}, we proved that a diffusion on $\Gamma$ can be expressed as a time-change of a Walsh Brownian motion on $\Gamma^{\ext}$, where the time-change is given by the quadratic variation of the distance-to-origin process $X$. To establish Theorem~\ref{thm_main}, we require an explicit representation of this time-change. This representation relies on the results of Section~\ref{sec_further} for Walsh Brownian motion and on a characterization of $H''_b$ (the function $H_b $ is introduced in  Notation~\ref{notation_exit_times}). 

First, we prove the characterization of $H''_b$. Next, we establish Theorem~\ref{thm_main} and use it to prove Theorem~\ref{thm_main2}.

\subsection{Characterization of $H''_b$}

Recall the definitions of $T_0$, $T_b$, and $H_b(\bx)$ for $b > 0$ and $\bx \in \Gamma$, as given in Notation~\ref{notation_exit_times}. For all $b > 0$, define $T_{0,b} := \min \{T_0, T_b\}$.

\begin{lemma}
	\label{lem_derivatives_EET}
	Let $\mX$ be a regular NSE diffusion on $\Gamma$, defined by~\ref{item_A1}--\ref{item_A2}. 
	Then, the following hold:
	\begin{enuroman}
		\item For all $ e \in E $ and $ x > 0 $,
		\begin{equation}
			- \frac{1}{2} H_{b}''(e, \rd x) ={} \m_{e}(\rd x) 
		\end{equation}
		\label{item_Hb_not0}
		\item The function $H_b $ admits a concave extension $\wt H_b $ such that 
		\begin{equation}
			- \sum_{e \in E} \beta_e \wt H''_{b}(e, \{0\}) = \rho.
		\end{equation}
		\label{item_Hb_at0}
	\end{enuroman} 
\end{lemma}

\begin{proof}	
	We first prove~\ref{item_Hb_not0}. 
	Let $\bx = (e,x) $, with $x>0 $. 
	From the Markov property and Bayes' rule, 
	\begin{align}
		H_b(\bx) 
		={}& \Esp_{\bx} \sqbraces{T_b}
		= \Esp_{\bx} \sqbraces{\indicB{T_0\le T_b} T_b}
		+ \Esp_{\bx} \sqbraces{\indicB{T_b < T_0} T_b},
		\\
		={}& \Esp_{\bx} \sqbraces{\indicB{T_0\le T_b} T_0}
		+ \Prob_{\bx}(T_0 < T_b) \Esp_{\bv} \sqbraces{T_b} 
		+ \Esp_0 \sqbraces{\indicB{T_b < T_0} T_b }
		\\
		={}& 
		2 \int_{(0,b)} G^{e}_{0,b}(x,y) \frac{b-y}{b} \m_e(\rd y)
		+ \frac{b-x}{b} \Esp_0 \sqbraces{\indicB{T_b < T_0} T_b }
		\\ & + 2 \int_{(0,b)} G^{e}_{0,b}(x,y) \frac{y}{b} \m_e(\rd y)
		\\
		={}& 
		2 \int_{(0,b)} G^{e}_{0,b}(x,y) \m_e(\rd y)
		+ \frac{b-x}{b} \Esp_0 \sqbraces{\indicB{T_b < T_0} T_b }
		\\
		={}& \Esp_{\bx} \sqbraces{T_{0,b}} + \frac{b-x}{b} \Esp_{\bv} \sqbraces{ T_b },
		\label{eq_proof_Hb_1}
	\end{align}
	where 
	\begin{equation}
		\begin{aligned}
			G^{e}_{a,b}(x,y)
			&:= \frac{((x\wedge y) - a)(b-(x\vee y))}{b-a},
			&  \forall\, x,y &\in (a,b)
			&  \forall\, & 0<a<b 
		\end{aligned}
	\end{equation}
	From Lebesgue convergence theorem, 
	\begin{align}
		\frac{1}{2} \D_x \Esp_{\bx} \sqbraces{T_{0,b}} & = \D_x \int_{(0,b)} \braces{\indicB{x<y} \frac{x(b-y)}{b} + \indicB{y\le x}
			\frac{y(b-x)}{b}} \m_e(\rd y)
		\\ &=
		\int_{(0,b)} \braces{\indicB{x<y} \frac{(b-y)}{b} - \indicB{y\le x}
			\frac{y}{b}} \m_e(\rd y)
		\\ &=
		\int_{(x,b)} \frac{b-y}{b} \m_e(\rd y)
		+ \int_{(0,x]} \frac{y}{b} \m_e(\rd y)
		\label{eq_proof_Hb_2}
	\end{align}
	and, for $x'>x $,
	\begin{align}
		- \frac{1}{2} &\braces{\D_x \Esp_{(e,x')} \sqbraces{T_{0,b}}
			-  \D_x \Esp_{(e,x)} \sqbraces{T_{0,b}}}
		\\ &=	\int_{(x,b)} \frac{b-y}{b} \m_e(\rd y)
		- \int_{(0,x]} \frac{y}{b} \m_{e}(\rd y)
		- 	\int_{(x',b)} \frac{b-y}{b} \m_e(\rd y)
		+ \int_{(0,x']} \frac{y}{b} \m_{e}(\rd y)
		\\
		& =  \int_{(x,x']} \frac{b-y}{b} \m_e(\rd y)
		+ \int_{(x,x']} \frac{y}{b} \m_e(\rd y)
		= \m_e((x,x'])
		\label{eq_proof_Hb_3}
	\end{align}
	From \eqref{eq_proof_Hb_1} and \eqref{eq_proof_Hb_3}, 
	\begin{equation}
		-\frac{1}{2} \braces{H_{b}'(e,x') - H_{b}'(e,x')}
		= \m_e((x,x']), 
	\end{equation}
	which proves~\ref{item_Hb_not0}. 
	
	We now prove~\ref{item_Hb_at0}. 
	We consider the extension $\wt H_b $ of $H_b $ 
	defined as
	\begin{equation}
		\label{eq_Hb_extension}
		\wt H_b(e,x) :=
		\begin{cases}
			H_b(e,x), & x\ge 0,\; e\in E,\\ 
			H_b(e,0), &x<0,\; e\in E.
		\end{cases}
	\end{equation}
	If $\bx = \bv $, from the Markov property, 
	and the definition of $(\beta^{\delta}_{\epsilon};\, e\in E,\, \delta >0) $,
	\begin{align}
		H_b(\bv)
		= H_{\delta}(\bv) + \sum_{e\in E} \Esp_{\bv} \sqbraces{ \braces{T_b \circ \theta_{T_{\delta}}} \indicB{I(T_{\delta})=e} } = H_{\delta}(\bv) + \sum_{e\in E} \beta_{\epsilon}^{\delta} H_b(e,\delta).
	\end{align}
	Hence, from Proposition~\ref{prop_node_infinitesimal}, 
	\begin{align}
		&\frac{1}{\delta} \sum_{e\in E} \beta^{\delta}_{e} \braces{H_b(\bv)-H_b(\epsilon,\delta)}
		= \frac{1}{\delta} H_{\delta}(\bv) 
		\\
		\Rightarrow \quad
		&\lim_{\delta\to 0} \frac{1}{\delta} \sum_{e\in E} \beta^{\delta}_{e} \braces{H_b(\bv)-H_b(\epsilon,\delta)}
		= \lim_{\delta\to 0}  \frac{1}{\delta} H_{\delta}(\bv) 
		\\
		\Rightarrow \quad
		&-\sum_{e\in E} \beta_{e} H'_b(e,0)
		= \rho
		\\
		\Rightarrow \quad
		& 
		\sum_{e\in E} \beta_{e} \wt H''_b(e,\{0\})
		= \sum_{e\in E} \beta_{e} \braces{\wt H'_b(e,0+) - \wt H'_b(e,0-)}
		=  - \rho. 
	\end{align}
	Obviously, the extension is concave on $\IR $.
	Hence,~\ref{item_Hb_at0} holds, which completes the proof. 
\end{proof}

\subsection{Proof of Theorem~\ref{thm_main}}

\begin{proof}
	[Proof of Theorem~\ref{thm_main}]
	We first suppose that $\mX $ is NSE. 
	From Theorem~\ref{thm_DDS_Walsh}, to complete the proof we need only to characterize the quadratic variation $\qv{X} $ in~\eqref{eq_thm_graph_DDS_Walsh}, in the fashion of the proof of~\cite[Theorem~V.47.1]{RogWilV2}.
	
	Let $(\wt A(t);\, t\ge 0 ) $ be the right-inverse of $\qv{X} $
	in the sense~\eqref{eq_def_right_inverse}. 
	Further, define $\wt T_b := \qv{X}_{T_b}$. 
	Consider the uniformly integrable martingale
	\begin{equation}
		\begin{aligned}
			M_t &:= \Esp_x \braces{ T_b \big| \bF_t }
			= H_b(I(t \wedge T_b),X_{t \wedge T_b}) + \braces{t\wedge T_b},
			& t& \ge 0.
		\end{aligned}
	\end{equation}
	Applying the time-change $\wt A $
	on $M$, we obtain the uniformly integrable martingale
	\begin{equation}
		\label{eq_proof_Mprime_def}
		\begin{aligned}
			M'_t 
			:= M_{\wt A(t)} 
			&= H_b(I(\wt A(t) \wedge T_b),X_{\wt A(t) \wedge T_b}) + \braces{\wt A(t)\wedge T_b}
			\\ &= H_b(J(t \wedge \wt T_b), R_{t \wedge \wt T_b}) + \braces{\wt A (t)\wedge T_b}
		\end{aligned}
	\end{equation}
	
	We first observe that for all $e\in E $, the function $[x \mapsto H_b(e,x)] $ is strictly concave.
	Indeed, for all $0\le x \le y $ and $e\in E $, we have that 
	\begin{align}
		H_b(e,\lambda x + (1-\lambda)y)
		={}& \Esp_{(e,\lambda x + (1 - \lambda)y)} \sqbraces{T_b} 
		\\ ={}& \Prob_{(e,\lambda x + (1 - \lambda)y)} \braces{T_{(e,x)} < T_{(e,y)}} \Esp_{(e,x)} \sqbraces{T_b \indicB{T_{(e,x)} < T_{(e,y)}}} 
		\\ &+ \Prob_{(e,\lambda x + (1 - \lambda)y)} \braces{T_{(e,x)} > T_{(e,y)}} 
		\Esp_{(e,y)} \sqbraces{T_b \indicB{T_{(e,x)} > T_{(e,y)}} } 
		\\
		\ge& \lambda \Esp_{(e,x)} \sqbraces{T_b} 
		+ (1-\lambda) \Esp_{(e,y)} \sqbraces{T_b}
		\\
		={}& \lambda H_b(e,x) + (1-\lambda) H_b(e,y).
	\end{align}
	
	We consider the extension $\wt H_b $ of $H_b $ defined in~\eqref{eq_Hb_extension}. 
	From~\eqref{eq_proof_Mprime_def} and Proposition~\ref{prop_ItoTanaka2},
	\begin{align}
		M'_{t}
		={}& 
		\wt H_b(I(0),R_0)
		+ \int_{0}^{t\wedge \wt T_b} (\wt H_b)'_-(J(s),R_{s}) \vd R_s
		\\& - \frac{1}{2} \sum_{e\in E} \int_{[0,\infty)} \int_{0}^{t \wedge \wt T_b} \indicB{J(s)=e} \vd \loct{R}{y+}{s} \, \wt H_b''(e,\rd y)
		+ \braces{\wt A (t)\wedge T_b}. 
	\end{align}
	From the representation of $R$ in~\eqref{eq_proof_repr_refl_2},
	\begin{align}
		\qv{M'}_{t} ={}&\sum_{e\in E}  \int_{0}^{t\wedge \wt T_b} (\wt H_b)'_- (e,0) \indicB{J(s)=e} \vd \loct{W}{0}{t}
		\\ &- \frac{1}{2} \sum_{e\in E} \int_{(0,\infty)} \int_{0}^{t \wedge \wt T_b} \indicB{J(s)=e} \vd \loct{R}{y+}{s} \, \wt H_b''(e,\rd y)
		\\ &- \frac{1}{2} \sum_{e \in E} 
		\beta_e \loct{R}{0+}{t \wedge \wt T_b} \, \wt H_{b}'' (e,\{0\}) + \braces{\wt A (t)\wedge T_b}.
	\end{align}
	Since $ (\wt H_b)'_-(e,0) =0 $ for all $e\in E $, we obtain that
	\begin{align}
		\wt A (t) \wedge T_b ={}&  \frac{1}{2} \sum_{e\in E} \int_{(0,\infty)} 
		\int_{0}^{t \wedge \wt T_b} \indicB{J(s)=e} \vd \loct{R}{y+}{s} \, \wt H_b''(e,\rd y)
		\\&+  \frac{1}{2} \braces{\sum_{e \in E} 
			\beta_e \wt H_{b}'' (e,\{0\})} \, \loct{R}{0+}{t \wedge \wt T_b} .
	\end{align}
	
	From Lemma~\ref{lem_derivatives_EET} and the definition of \( A \), we have  
	\begin{equation}
		\label{eq_AAtilde_relation}
		\wt A (t) \wedge T_b =  \sum_{e\in E} \int_{(0,\infty)} \int_{0}^{t \wedge \wt T_b} \indicB{J(s)=e} \vd \loct{R}{y+}{s} \, \m_e(\rd y)
		+ \frac{\rho}{2} \loct{R}{0+}{t \wedge \wt T_b}
		= A(t \wedge \wt T_b).
	\end{equation}
	
	Let \(\zeta := \lim_{t\to \infty} \qv{X}_t\). Since \(T_b \to \infty\) almost surely as \(b \to \infty\), taking the limit in~\eqref{eq_AAtilde_relation} yields  
	\begin{equation}  
		\wt A(t) = A(t), \quad \forall t \in [0, \zeta).  
	\end{equation}  
	
	We observe that \(\wt A\) is strictly increasing, as it is the right-inverse of the increasing continuous process \(\qv{X}\). Consequently, \(\wt A = A: [0, \zeta) \mapsto [0, \infty)\) is a homeomorphism, and \(\qv{X}_t = \gamma(t)\) for all \(t \ge 0\).  
	
	For the converse assertion, by \cite[Theorem~10.10]{Dyn1}, the process $\mX$ is a diffusion on $\Gamma$ adapted to the filtration $(\bF_{\gamma(t)};\, t \ge 0)$. 
	Its generator is uniquely determined by the first part of the proof,
	therefore $\mX $ is defined~\ref{item_A1}--\ref{item_A2}.
	
	This completes the proof for a regular NSE diffusion on \(\Gamma\). The general case follows from Proposition~\ref{prop_change_of_scale}.
\end{proof}

\subsection{Proof of Theorem~\ref{thm_main2}}

\begin{proof}[Proof of Theorem~\ref{thm_main2}]
	By Theorem~\ref{thm_main}, there exists a Walsh Brownian motion $\mW = (J, R)$
	on $\Gamma^{\ext} $, of bias parameters $(\beta_e;\, e \in E)$, defined on an extension of $\mc P_{\bx}$, adapted to the filtration $\process{\bF_{\gamma(t)}}$, such that $\mX_t = q(\mY_t)$, where $\mY_t = \mW_{\gamma(t)}$ for all $t \ge 0$. Here, $\gamma$ is the right-inverse of the process
	\begin{equation}
		\begin{aligned}
			A(t) &:= \sum_{e \in E} \int_{0}^{t} \indicB{J(s) = e} \int_{(0, \infty)} \loct{R}{y+}{s} \, \m^{\mY}_e(\rd y) \vd s 
			+ \frac{\rho^{\mY}}{2} \loct{R}{0+}{t}, 
			& \forall\, t &\ge 0.
		\end{aligned}
	\end{equation}
	
	Let $\gamma^{\circ}$ be the right-inverse of
	\begin{equation}
		\begin{aligned}
			A^{\circ}(t) &:= \sum_{e \in E} \int_{0}^{t} \indicB{J(s) = e} \int_{(0, \infty)} \loct{R}{y}{s} \, \m^{\mY}_e(\rd y) \vd s, 
			& \forall\, t &\ge 0.
		\end{aligned}
	\end{equation}
	
	Define the process $\mX^{\circ}$ as $\mX^{\circ} := q(\mY^{\circ})$, where $\mY^{\circ} := (\mW_{\gamma^{\circ}(t)};\, t \ge 0)$, and admit the representation $\mY^{\circ} = (Y^{\circ}, I^{\circ})$. We observe that
	\begin{equation}
		\mX^{\circ}_t = q(\mY^{\circ}_t) 
		= q(\mW_{\gamma^{\circ}(t)})
		= q(s(\mX_{A \circ \gamma^{\circ}(t)}))
		= \mX_{A \circ \gamma^{\circ}(t)}, 
		\quad \forall\, t \ge 0.
	\end{equation}
	
	By the second part of Theorem~\ref{thm_main}, the process $\mX^{\circ}$ is a diffusion adapted to the filtration $(\bF_{\gamma \circ A^{\circ}(t)};\, t \ge 0)$, defined by~~\ref{item_A1} and~\eqref{eq_def_lateral2}. To complete the proof, it remains to show that $\gamma_{\rho} = A^{\circ} \circ \gamma$ and $A_{\rho} = \gamma \circ A^{\circ}$.
	
	By Lemma~\ref{def_TC_local_time},
	\begin{align}
		A_{\rho} \circ A^{\circ}(t)
		&= \sum_{e \in E} \int_{0}^{t} \indicB{J(s) = e} \int_{(0, \infty)} \loct{R}{y}{s} \, \m^{\mY}_e(\rd y) \vd s
		+ \frac{\rho^{\mY}}{2} \loct{Y^{\circ}}{0}{A^{\circ}(t)}
		\\
		&= \sum_{e \in E} \int_{0}^{t} \indicB{J(s) = e} \int_{(0, \infty)} \loct{R}{y}{s} \, \m^{\mY}_e(\rd y) \vd s
		+ \frac{\rho^{\mY}}{2} \loct{R}{0}{t}
		= A(t), \quad \forall\, t \ge 0.
	\end{align}
	
	Consequently, we have $A = A_{\rho} \circ A^{\circ}$, $\gamma = \gamma^{\circ} \circ \gamma_{\rho}$, $\gamma_{\rho} = A^{\circ} \circ \gamma$, and $A_{\rho} = A \circ \gamma^{\circ}$. This proves the first assertion.
	
	Regarding the second assertion, by \cite[Theorem~10.10]{Dyn1}, the process $\mX $ is a diffusion on $\Gamma $
	adapted to the filtration $(\bF_{\gamma_{\rho}(t)};\, t\ge 0) $.
	Its generator is uniquely determined by Theorem~\ref{thm_main2}.
	This completes the proof. 
\end{proof}

%% New Section
\section{A second occupation times formula for $\mX $}
\label{sec_occupation_II}

In this section, we derive a homogeneous occupation times formula for $\mX$, analogous to the one  in~\cite[Section~5.4]{ItoMcKean96} for one-dimensional diffusions on natural scale. From this formula, we establish a relationship between occupation time and local time at the junction vertex $\bv$. This aligns with analogous results in \cite{Anagnostakis2022,EngPes,Sal2017} for one-dimensional diffusions.

We note that the formula in~\cite[Section~5.4]{ItoMcKean96} is mistakenly recalled in~\cite[\S II.2.13]{BorSal} without mentioning the requirement of natural scale. 

\begin{proposition}
	[Second Occupation Times Formula]
	\label{prop_occupation_ItoMcKean}
	Let $\mX$ be a \emph{Natural Scale on Edges} (NSE) diffusion on $\Gamma$, defined by~\ref{item_A1}--\ref{item_A2}. 
	For any measurable real-valued function $f$ on $\Gamma$, the following holds:
	\begin{equation}
		\int_{0}^{t} f(\mX_s) \vd s
		= \int_{(0,\infty)} \sum_{e\in E} \braces{\int_{0}^{t} \indicB{I(s)=e} \vd \loct{X}{y+}{s} \vd s} f_{e}(y) \, \m_{e}(\rd y) 
		 + \frac{\rho}{2} f(\bv) \loct{X}{0+}{s}, 
	\end{equation}
	for all $t\ge 0 $.
\end{proposition}

\begin{proof}
	Assume $\mX$ is defined on the probability space $\mc P_{\bx} := (\Omega, \process{\bF_t}, \Prob_{\bx})$. By Theorem~\ref{thm_main}, there exists a $(\beta_e;\, e\in E)$-Walsh Brownian motion $\mW := (J, R)$, adapted to $\process{\bF_{A(t)}}$, defined on an extension of $\mc P_{\bx}$, such that $\mX_t = \mW_{\gamma(t)}$, for all $t \ge 0$. Here, $A$ is the time change defined in~\eqref{eq_thm_A_timechange}, and $\gamma$ is its right inverse.
	
	From Lemmas~\ref{def_TC_stochastic_integral} and~\ref{def_TC_local_time}, we have 
	\begin{align}
		\int_{0}^{t} f(\mX_s) \vd s={}&
		\sum_{e\in E} \int_{0}^{t} f_e(X_s) \indicB{I(s)=e} \vd s
		\\={}& \sum_{e\in E} \int_{0}^{t} f_e(R_{\gamma(s)}) \indicB{J(\gamma(s))=e} \vd s
		= \sum_{e\in E} \int_{0}^{\gamma(t)} f_e(R_{s}) \indicB{J(s)=e} \vd A(s)
		\\={}& \sum_{e\in E} \int_{0}^{\gamma(t)} f_e(R_{s}) \indicB{J(s)=e} 
		\, \braces{\frac{\rho}{2} \vd \loct{R}{0+}{s} 
			+ \int_{(0,\infty)} \vd \loct{R}{y+}{s}  \,\m_e(\rd y)}
		\\={}& \sum_{e\in E} \int_{0}^{t} f_e(X_{s}) \indicB{I(s)=e} 
		\, \braces{\frac{\rho}{2} \vd \loct{X}{0+}{s} 
			+ \int_{(0,\infty)} \vd \loct{X}{y+}{s}  \,\m_e(\rd y)}. 
	\end{align}
	This completes the proof.
\end{proof}

\begin{remark}
	\label{rmk_loct_occt}
	Taking \(f = \indicB{\bx=\bv}\) in the above formula yields the relation:
	\begin{equation}
		\label{eq_relation_occt_loct0}
		\int^{t}_{0} \indicB{X_s = 0} \vd s = \frac{\rho}{2} \loct{X}{0+}{t}, \quad \forall\, t \ge 0.
	\end{equation}
\end{remark}

\appendix

%% New Appendix Section
\section{Boundary classification}
\label{app_boundary}

Let us present the problem of classifying boundaries of a diffusion on the star graph, which reduces to the one dimensional case.
The boundary classification is formulated in terms of the following quantities: 
\begin{equation}
	I^{e}_1(x_0,x) := \abs{\int_{x_0}^{x} \int_{y}^{x} \m_e(\rd x)\, \s_e(\rd y)}
	\quad \text{and}\quad  I^{e}_2(x_0,x) := \abs{\int_{x_0}^{x} \int_{y}^{x} \s_e(\rd x) \, \m_e(\rd y)}, 
\end{equation}
for all $0 \le x_0 < x \le l_e$ and $e\in E$.

\begin{proposition}
	\label{prop_boundary_classification}
	Let $\mX $ be a regular diffusion on $\Gamma $. 
	The boundary behavior of $\mX $ can be classified as follows.
	\begin{enumerate}
		\item If $I^{e}_1(0,l_e) < \infty $, $I^{e}_2(0,l_e) < \infty$, then $(e,l_e) $
		is a regular boundary for $\mX $. This means that $(e,l_e) $ is accessible from the interior of $\Gamma $ and the interior of $\Gamma $ is accessible
		from $(e,l_e) $, i.e.,
		\begin{equation}
			\begin{aligned}
				\lim_{t\to \infty} \inf_{y > x} \Prob_{(e,y)} \braces{T_{\bx} \le t} &>0,
				&\text{for all } \bx = (e,x) \text{ with } x \in [0,l_e).
			\end{aligned}
		\end{equation}
		\item If $I^{e}_1(0,l_e) < \infty $, $I^{e}_2(0,l_e) = \infty$, then $(e,l_e) $
		is a exit boundary for $\mX $. This means that $(e,l_e) $ is accessible from the interior of $\Gamma $ and the interior of $\Gamma $ is inaccessible
		from $(e,l_e) $, i.e.,
		\begin{equation}
			\begin{aligned}
				\lim_{t\to \infty} \inf_{y > x} \Prob_{(e,y)} \braces{T_{\bx} \le t} &=0,
				&\text{for all } \bx = (e,x) \text{ with } x \in [0,l_e).
			\end{aligned}
		\end{equation}
		\item If $I^{e}_1(0,l_e) = \infty $, $I^{e}_2(0,l_e) < \infty$, then $(e,l_e) $
		is a entry boundary for $\mX $. This means that  the $(e,l_e) $ is inaccessible from the interior of $\Gamma $ and the interior of $\Gamma $ is accessible
		from $(e,l_e) $.
		\item If $I^{e}_1(0,l_e) = \infty $, $I^{e}_2(0,l_e) = \infty$, then $(e,l_e) $
		is a natural boundary for $\mX $. This means that $(e,l_e) $ is inaccessible from the interior of $\Gamma $ and the interior of $\Gamma $ is inaccessible
		from $(e,l_e) $,
	\end{enumerate}
\end{proposition}

\begin{proof}
	From the local character of $(\s_e,\m_e)$, for each $e\in E $, one can reduce by localization the 
	problem to the one-dimensional case (see~\cite[Section~5.11]{Ito2006}). 
	This completes the proof. 
\end{proof}

For a discussion on boundary classification of an NSE diffusion at a finite boundary,
see \cite[Section~5.3]{ItoMcKean96}.

%% New Appendix Section
\section{Regularity implies Feller property}
\label{app_Fprocess}

\begin{proposition}
	\label{prop_Feller_property}
	Every regular diffusion $\mX $ on $\Gamma $ is Feller.
\end{proposition}

The following proof is an adaptation of the proof of \cite[Proposition~V.50.1]{RogWilV2}. %p.291

\begin{proof}
	Let $\mX$ be defined on the probability space $(\Omega, \process{\bF_t}, \Prob_{\bx})$, such that $X_0 = \bx$ holds $\Prob_{\bx}$-almost surely.
	
	Let $f \in C_b(\Gamma)$, $t > 0$, and $\bx := (e, x) \in \Gamma$. The boundedness of $P_t f$ follows directly from H\"older's inequality. It thus remains to show that $P_t f$ is continuous.
	
	From continuity of $f $, for all $\varepsilon>0 $, there exists $\delta>0 $ such that 
	\begin{equation}
		\label{eq_proof_Feller_1}
		\begin{aligned}
			|P_t f(\bx) - P_s f(\bx)|&< \frac{\varepsilon}{2},
			& \forall\,&|t-s|< \delta.
		\end{aligned}
	\end{equation}
	
	For $y>x $, clearly the application $ y \mapsto \Prob_{(e,y)} (T_{(e,x)}> \delta)$
	is increasing and from dominated convergence, converges to $0$ as $y \to x $. 
	Hence, there exists some $y_0 > x $ such that 
	\begin{equation}
		\label{eq_proof_Feller_2}
		\begin{aligned}
			\Prob_{(e,y)} (T_{(e,x)}> \delta) &< \frac{\varepsilon}{2},
			& \forall\, y&\in (x,y_0).
		\end{aligned}
	\end{equation}
	
	From the strong Markov property, \eqref{eq_proof_Feller_1} and~\eqref{eq_proof_Feller_2},
	\begin{align}
		|P_t f(\by) - P_t f(\bx)| ={}&
		\abs{\Esp_{\by} \sqbraces{f(\mX_t) - P_t f(\bx)} } 
		\\ 
		\le& 
		\abs{\Esp_{\by} \sqbraces{ \indicB{T_{\bx}< \delta} \braces{f(\mX_t) - P_t f(\bx)}} }
		+ \abs{\Esp_{\by} \sqbraces{ \indicB{T_{\bx}< \delta} \braces{f(\mX_t) - P_t f(\bx)}} }
		\\
		< &
		\abs{\Esp_{\by} \sqbraces{ \indicB{T_{\bx}< \delta} \braces{P_{t-T_{\bx}}f(\bx) - P_t f(\bx)}} } + \varepsilon/2 < \varepsilon. 
	\end{align}
	The identity $ \Esp_{\by} \sqbraces{ \indicB{T_{\bx}< \delta} f(\mX_t)}
	= \Esp_{\by} \sqbraces{ \indicB{T_{\bx}< \delta} P_{t-T_{\bx}}f(\bx)} $ used above is justified in the proof of
	\cite[Proposition~50.1]{RogWilV2}
	with an approximation argument. 
	
	For $y< x $ we may apply a similar argument.
	This completes the proof. 
\end{proof}

%% New Appendix Section
\section{Proof of Proposition~\ref{prop_FD}}
\label{app_FD}

For reader's convenience, we recall the statement of Proposition~\ref{prop_FD} 

\begin{propositionOhne}
	A regular diffusion $\mX $ on $\Gamma $ is Feller--Dynkin if and only if all 
	open boundaries of $\Gamma $ are natural. 
\end{propositionOhne}

For the proof, we assume that $\mX $ is defined on the probability space
$(\Omega,\process{\bF_t},\Prob_{\bx}) $, such that $X_0 = \bx $, 
$\Prob_{\bx}$-almost surely.
We also consider the following condition:  
\begin{equation}
	\label{eq_FD_condition}
	\begin{aligned}
		\lim_{x \to \infty} \Esp_{(e,x)} \sqbraces{ e^{-\lambda \tau_{\by}}}
		&= 0,
		& \forall\,& \lambda > 0,
		& \forall\,& e \in E,
		& \forall\,& \by \in \Gamma.
	\end{aligned}
\end{equation}  
In the context of one-dimensional diffusions on natural scale, this condition is equivalent to either:  
\begin{enumerate}
	\item the process is Feller--Dynkin (see \cite[Lemma~2.1]{criens2023feller}), or  
	\item all infinite boundaries are natural (see \cite[Lemma~2.2]{criens2023feller}).  
\end{enumerate}

As intermediate result, we extend this equivalence to diffusions on \(\Gamma\). Specifically, we establish the following result:

\begin{lemma}
	\label{lem_FD_condition}
	Let \(\mX\) be a regular NSE diffusion on \(\Gamma\). Then:  
	\begin{enuroman}
		\item \(\mX\) is Feller--Dynkin if and only if condition~\eqref{eq_FD_condition} is satisfied. \label{item_FD_FD}  
		\item All infinite boundaries of \(\mX\) are natural if and only if condition~\eqref{eq_FD_condition} is satisfied. \label{item_FD_natural}  
	\end{enuroman}
\end{lemma}

Since the diffusion is regular, from Proposition~\ref{prop_Feller_property}, it is also Feller.
Hence, to prove Feller--Dynkin property, it is enough to show that
\begin{equation}
	\begin{aligned}
		\lim_{x\to \infty} P_t f(e,x) &= 0,
		& \forall\, e&\in E. 
	\end{aligned}
\end{equation} 

The proof of part~\ref{item_FD_FD} involves a straightforward adaptation of the argument in \cite[Lemma~2.1]{criens2023feller}. Part~\ref{item_FD_natural}, on the other hand, is a direct consequence of \cite[Lemma~2.2]{criens2023feller}.

\begin{proof}
	We first prove~\ref{item_FD_FD}.
	We assume that $\mX $ is a Feller--Dynkin process.
	Let $\by = (e,y) \in \Gamma $ and $g$ a function in $C_0(\Gamma) $
	such that $ g(\Gamma) \subset [0,1]$ and $g(y)=1 $.
	Let $R_{\lambda}g $ be the function defined as 
	\begin{equation}
		R_{\lambda}g = \int_{0}^{\infty} e^{-\lambda s} P_s g \vd s. 
	\end{equation}
	We have that $R_{\lambda}g \in C_0(\Gamma) $ (see~\cite[Theorem~17.4]{Kal}),
	and $(e^{-\lambda t} R_{\lambda}(\mX_t);\, t\ge 0) $ is a
	$\Prob_{\bx} $ supermartingale, for all $\bx \in \Gamma $ (see~\cite[Proposition~III.2.6]{RevYor}). 
	Moreover, $t \mapsto P_t $ is continuous at the origin and 
	$R_{\lambda} g(\by) >0 $.
	From optional stopping,
	\begin{equation}
		R_{\lambda} g(e,x) \ge 
		\Esp_{\bx} \sqbraces{ e^{-\lambda \tau_{\by}} R_{\lambda} g(\mX_{\tau_{\by}}) \indicB{\tau_{\by}< \infty} }
		= R_{\lambda}g(\by) \Esp_{(e,x)} \sqbraces{ e^{-\lambda \tau_{\by}} }.
	\end{equation}
	Since $R_{\lambda} g $ is $C_0(\Gamma) $, sending $x\to \infty $ results in~\eqref{eq_FD_condition}. 
	
	Let's assume that~\eqref{eq_FD_condition} holds.
	By virtue of \cite[Proposition~III.2.4]{RevYor}, it suffices to show that
	$P_t f $ vanishes at infinity, for every $f\in C_0(\Gamma) $ and $t>0 $. 
	From the Markov inequality, for $y< x $, $\bx = (e,x) $ and $\by = (e,y) $, 
	\begin{equation}
		\label{eq_proof_Markov_inq}
		\Prob_{\bx} \braces{X_t < y}
		\le \Prob_{\bx} \braces{\tau_{\by} < t}
		\le e^{a^{2}} \Esp_{\bx} \sqbraces{e^{-a \tau_{\by}}}. 
	\end{equation}
	Let $z >0$ be chosen such that $f(e,y) < \varepsilon $
	for all $e\in E $ and all $y> z $. 
	For such $z$, from~\eqref{eq_proof_Markov_inq},
	\begin{align}
		|P_t f(e,x)| 
		&\le \Esp_{(e,x)} \sqbraces{|f(\mX_t)| \indicB{X_t \ge z} }
		+ \Esp_{(e,x)} \sqbraces{|f(\mX_t)| \indicB{X_t < z} }
		\\
		&\le \varepsilon + \xnorm{f}{\infty} \Prob_{(e,x)} \braces{X_t < z}
		\le \varepsilon +  e^{a^{2}} \xnorm{f}{\infty} \Esp_{\bx} \sqbraces{e^{-a \tau_{\by}}}, 
	\end{align}
	which from~\eqref{eq_FD_condition} converges to $\varepsilon $ as $x\to \infty $.
	This means that $P_t f(e,x) \to 0$ as $x\to \infty $.
	This completes the proof of~\ref{item_FD_FD}.
	
	We now prove~\ref{eq_FD_condition}, which is direct from the local character of a diffusion.
	Indeed, we first observe that it suffices to prove that
	\begin{equation}
		\begin{aligned}
			\lim_{x\to \infty} \Esp_{(e,x)} \sqbraces{ e^{-\lambda \tau_{(e,y)}}}
			&= 0,
			& \forall\,& \lambda >0,
			& \forall\,& e \in E,
			& \forall\, y &\ge 0.
		\end{aligned}
	\end{equation}
	The reason is that
	\begin{equation}
		\begin{aligned}
			\Esp_{(e,x)} \sqbraces{ e^{-\lambda \tau_{(e',y)}}}
			&\le
			\Esp_{(e,x)} \sqbraces{ e^{-\lambda \tau_{\bv}}},
			& \forall\, \lambda&>0, 
			& \forall\, e' &\not = e,
			& \forall\, y &>0.
		\end{aligned}
	\end{equation}
	By localization, the law of the process $\mX $ started at $(x,e) $ and stopped at $(y,e)$, with $0\le y < x $, is the one of a stopped
	one dimensional diffusion.
	Hence, assertion~\ref{item_FD_natural} follows from \cite[Lemma~2.2]{criens2023feller}.
	This completes the proof.
\end{proof}

\begin{proof}
	[Proof of Proposition~\ref{prop_FD}]
	We first extend the observation in the opening  
	of~\cite[Section~2]{criens2023feller} to $\mX $. 
	The function $s $ is a homeomorphism from $\Gamma $ to $s(\Gamma) $ and
	from Proposition~\ref{prop_change_of_scale}, $s(\mX) $ is a regular
	NSE diffusion on $\Gamma^{*} := s(\Gamma) $. 
	We observe that $f\in C_0(\Gamma) \Rightarrow f \circ q \in C_0 (\Gamma^{*}) $
	and that $f\in C_0(\Gamma^{*}) \Rightarrow f \circ s \in C_0 (\Gamma) $. 
	Hence $\mX $ is Feller--Dynkin if and only if $s(\mX) $ is Feller--Dynkin. 
	
	We also recall that for a one dimensional diffusion on natural scale,
	every finite open boundary is natural (see \cite[Proposition~16.45]{breiman1992probability}).
	By localization, this holds also for any NSE diffusion on $\Gamma $. 
	
	With the above, the proof follows from the two elements of Lemma~\ref{lem_FD_condition}
	and consists in a change of scale. 
\end{proof}

%% New Appendix Section
\section{Proof of Proposition~\ref{prop_strict_monotony}}
\label{app_strict_monotony}

For the reader's convenience, we restate the result. 

\begin{propositionOhne}
	Let \(\mX = (I, X)\) be a regular NSE diffusion on \(\Gamma\), \(\qv{X}\) be the quadratic variation of \(X\), and \(\varphi\) the right-inverse of \(t \mapsto \qv{X}_t\) (see~\eqref{eq_def_right_inverse} for the definition of right-inverse). 
	The mappings \(t \mapsto \qv{X}_t\) and \(t \mapsto \varphi(t)\) are almost surely strictly increasing. 
	Consequently, \(\qv{X}\) and \(\varphi\) are proper inverses, i.e., 
	\[
	\varphi(\qv{X}_t) = \qv{X}_{\varphi(t)} = t, \quad \text{for all } t \ge 0.
	\]
\end{propositionOhne}

To prove the result we will first need a preliminary lemma that extends  \cite[Proposition~IV.1.13]{RevYor} to other processes that are not local martingales.

\begin{lemma}
	\label{lem_qv_null}
	Let $\mX = (I,X) $ be a regular NSE diffusion on $\Gamma $. 
	For almost all $\omega \in \Omega $,
	\begin{equation}
		\begin{aligned}
			X_u(\omega) &= X_s(\omega), 
			& \forall\,& u\in [s,t]
			& \Leftrightarrow&
			& \qv{X}_t(\omega) &= \qv{X}_s (\omega).
		\end{aligned}
	\end{equation}
\end{lemma}

\begin{proof}
	Inspired by arguments in the proof of Lemma~\ref{thm_graph_DDS}, we define the local martingale $Z $ as
	\begin{equation}
		\begin{aligned}
			Z_t &:= X_0 + \int^{t}_{0} \indicB{X_s>0} \vd X_s = X_t - \frac{1}{2} \loct{X}{0}{t},
			& t&\ge 0.
		\end{aligned}
	\end{equation}
	
	We first observe that $\qv{Z} = \qv{X} $.
	Since $Z$ is a local martingale, from \cite[Proposition~IV.1.13]{RevYor}, for almost all $\omega \in \Omega$, the mapping $t \mapsto Z_t(\omega)$ is constant on $[a, b]$ if and only if the mapping $t \mapsto \qv{Z}_t(\omega)$ is constant on $[a, b]$. 
	
	We also observe that almost surely,  
	\begin{equation}
		\cubraces{ X_u = X_s;\; \forall\, u \in [s, t] }
		= \cubraces{ Z_u = Z_s;\; \forall\, u \in [s, t] }.
	\end{equation}
	This is justified by the following. 
	\begin{itemize}
		\item If $X_s>0 $ and $X$ is constant on $[s,t] $, then 
		$\loct{X}{0}{t} - \loct{X}{0}{s} = 0 $. Consequently, $Z_u - Z_s = X_u - X_s $, for all $u\in [s,t] $.
		\item If $X_s>0 $ and $Z$ is constant on $[s,t] $, then 
		$ \vd X_u  = \frac{1}{2} \vd \loct{X}{0}{u} $ for all $u\in [s,t] $.
		This means that $ X$ is of bounded variation on $[s,t] $.
		Since it is locally a regular one-dimensional diffusion on natural scale, 
		necessarily $X $ is constant on $[s,t] $. 
		Consequently, $Z_u - Z_s = X_u - X_s $, for all $u\in [s,t] $.
		\item If $x=0 $ and either $X$ or $Z$ is constant, then
		\begin{equation}
			\loct{X}{0}{u} - \loct{X}{0}{s} 
			=
			\lim_{\delta \to 0} \int^{u}_{s} \indicB{X_{\zeta} \in [0,\delta)} \vd \qv{X}_\zeta
			=
			\lim_{\delta \to 0} \int^{u}_{s} \indicB{X_{\zeta} \in [0,\delta)} \vd \qv{Z}_\zeta
			= 0.
		\end{equation}
		Consequently, $Z_u - Z_s = X_u - X_s $ for all $u\in [s,t] $.
	\end{itemize}
	Combining the above completes the proof.
\end{proof}

\begin{proof}
	[Proof of Proposition~\ref{prop_strict_monotony}]
	We first analyze the mapping \( t \mapsto \qv{X}_t \). By Lemma~\ref{lem_qv_null}, we have
	\begin{equation}
		\label{eq_condition_null_qv2}
		\Prob_{\bx} \braces{ \qv{X}_t = \qv{X}_s } = \Prob_{\bx} \braces{ X_u = X_s;\; \forall\, u \in [s, t] }.
	\end{equation}
	From the Markov property,
	\begin{equation}
		\Prob_{\bx} \braces{ \qv{X}_t = 0 } = \Prob_{\bx}  \braces{ X_u = x;\; \forall\, u \in [0, t] } = 0,\quad \text{for all } t > 0 \text{ and } \bx = (e,x) \in \Gamma.
	\end{equation}
	
	Assume, for contradiction, that the mapping $t \mapsto \qv{X}_t $ is not almost surely strictly increasing. 
	Then, there exists $t > 0$ and $\bx \in \Gamma$ such that
	\begin{equation}
		\label{eq_proof_assumption_0}
		\Prob_{\bx} \braces{ \qv{X}_t = 0 } = \Prob_{\bx} \braces{ X_u = x;\; \forall\, u \in [0, t] } > 0.
	\end{equation}
	
	Let \( U \) be a relatively compact open set in \( \Gamma \) containing \( \bx \), and define the stopping time \( T_{U^{c}} := \inf\{t \geq 0:\; X_t \notin U\} \). Consider the stopped process \( X^U := (X_{t \wedge T_{U^c}};\, t \geq 0) \). By the continuity of \( X \) and \( \qv{X} \), we have
	\begin{equation}
		\label{eq_proof_assumption_0_stopped}
		\Prob_{\bx} \braces{ \qv{X^U}_t = 0 } = \Prob_{\bx} \braces{ X^U_u = x;\; \forall\, u \in [0, t] } > 0.
	\end{equation}
	
	By Proposition~\ref{prop_FD}, for sufficiently small \( U \), the process \( X^U \) is a Feller--Dynkin process. Now, consider a sequence \( (\bx_n)_n \) in \( \Gamma \) such that \( \lim_{n \to \infty} \bx_n = \bx \). Define the stopping times
	\[
	T_n := \inf\{t \geq 0:\; \mX_t = \bx_n\} \wedge T_{U^{c}}, \quad n \in \IN.
	\]
	Since \( X^U \) is Feller--Dynkin, Blumenthal's zero-one law (see \cite[Corollary~17.18]{Kal}) implies that \( T_n \to 0 \) almost surely as \( n \to \infty \). Consequently, \( T_n \to 0 \) in probability as well. This contradicts \eqref{eq_proof_assumption_0_stopped}, and hence \eqref{eq_proof_assumption_0} cannot hold. Therefore, \( t \mapsto \qv{X}_t \) is almost surely strictly increasing. 
	
	Regarding the mapping \(t \mapsto \varphi(t)\), it is the right-inverse of the almost surely continuous and increasing process \(\qv{X}\). 
	Consequently, \(\varphi\) is almost surely strictly increasing. 
	This completes the proof.
\end{proof}

%% New Appendix Section
\section{Solution to Dirichlet problem and stickiness}
\label{app_Dirichlet}

In this section, we prove two results. 
First, we establish a generalized second order operator version of~\cite[Proposition~A.1]{berry2024sticky}, namely the existence, uniqueness, and Green function for the corresponding Dirichlet problem on the unit disc. 
Based on this result, we deduce the asymptotic behavior of the solution at the junction vertex $\bv$, using its explicit form. 
These results are leveraged to prove Proposition~\ref{prop_node_infinitesimal}, which characterizes the behavior of diffusion at $\bv $. 

\begin{proposition}
	\label{prop_Dirichlet_solution}
	For all $\delta\in (0,l_*)$ and $f \in C_b(\Gamma) $, we consider the
	Dirichlet problem
	\begin{equation}
		\label{eq_boundary_value_problem}
		\begin{cases}
			\frac{1}{2}\D_{\m_{e}} \D_{\s_e} u(e,x) = f(e,x), &\forall\, x \in (0,\delta),\quad \forall\, e\in E,\\
			\rho \Lop_e u(e,0) = \sum_{e'\in E} \beta_{e'} u_{e'}'(0),
			&\forall\, e\in E,\\
			u(e,\delta) = 0, &\forall\, e\in E, \\
			u(e,0) = u(e',0), &\forall\, e,e'\in E. 
		\end{cases}
	\end{equation}
	\begin{enuroman}
		\item \label{item_existence} A solution to~\eqref{eq_boundary_value_problem} is
		\begin{equation}
			\label{eq_Dirichlet_sol_form}
			\begin{aligned}
				u(e,x)& := b_e(x)  -b_e(\delta) - \frac{\rho f(\bv)}{ \s'_e(0)} \braces{\s_e(x) -  \s_e(\delta)},
			\end{aligned}
		\end{equation} 
		where $(b_e;\;e\in E) $ is the family of functions 
		defined as 
		\begin{equation}
			\begin{aligned}
				b_e(x) &:= 2 \int_{(x,\delta)} \int_{(0,y)} f_{e}(\zeta) \,\m_{e}(\rd \zeta) \, \s_e(\rd y),
				& \forall\,&e\in E,
				& \forall\,&x \in [0,\delta].
			\end{aligned}
		\end{equation}
		\item \label{item_uniqueness} The problem~\eqref{eq_boundary_value_problem} has a unique solution in 
		$\bigoplus_{e\in E}C^{\Lop_{e}}([0,\delta]) $.
	\end{enuroman}
\end{proposition}

\begin{proof}
	Proof of~\ref{item_existence}.
	We consider the general form of the solution
	\begin{equation}
		u(e,x) = b_e(x) + A_e \s_e(x) + B_e,
	\end{equation}
	with $(A_e;\, e\in E) $, $(B_e;\, e\in E) $ two families of
	real constants to be determined.
	
	For $x\in (0,\delta) $ and $e\in E $, we have that 
	\begin{equation}
		\Lop_e u(e,x) = \frac{1}{2} \D_{\m_e} \D_{\s_e} u(e,x)
		= - \frac 12 \D_{\m_e} \braces{ 2 \int_{(0,x)} f_{e}(\zeta) \,\m_{e}(\rd \zeta) }
		= - f_e(x).
	\end{equation}
	
	The boundary conditions (for $x=\delta $) yields that
	\begin{equation}
		B_e := -b_e(\delta) - A_e \s_e(\delta),
	\end{equation}
	and therefore
	\begin{equation}
		u(e,x) = b_e(x) + A_e \s_e(x) -b_e(\delta) - A_e \s_e(\delta) 
	\end{equation}
	
	At $x=0 $, we must have
	\begin{enumerate}
		\item $ u(e,0)=u(e',0)$,
		for all $ e,e'\in E$, 
		\item $\rho \Lop_e u(e,0) = \sum_{e'\in E} \beta_{e'} u_{e}'(0)$,
		for all $e\in E $. 
	\end{enumerate}
	From the first assertion, we have that 
	\begin{equation}
		u'(e,0) = A_e \s_e'(0) + \braces{\lim_{x\to 0}\int_{(0,x)} f_{e}(\zeta) \,\m_{e}(\rd \zeta)}
		= A_e \s_e'(0). 
	\end{equation}
	This entails that $A_e \s'_e(0) = A_{e'} \s'_{e'}(0) $, for all $e,e'\in E $. 
	From the second assertion, we have that
	\begin{equation}
		-\rho f(\bv) = \sum_{e\in E} \beta_e A_e \s_e'(0)
		= A_e \s_e'(0) = A_{e'} \s_{e'}'(0),
	\end{equation}
	for all $e,e'\in E $.
	This entails that $A_e = -\rho f(\bv)/ \s'_e(0) $.
	To sum up
	\begin{equation}
		u(e,x) = b_e(x)  -b_e(\delta) - \frac{\rho f(\bv)}{ \s'_e(0)} \braces{\s_e(x) -  \s_e(\delta)}. 
	\end{equation}
	
	Hence, the function $u $ defined above solves~\eqref{eq_boundary_value_problem}, which proves the existence of a solution.
	
	Proof of~\ref{item_uniqueness}, i.e., uniqueness of the solution.
	From linearity, it is enough to prove that $u=0 $ is the unique solution of the problem 
	\begin{equation}
		\label{eq_boundary_value_problem_0}
		\begin{cases}
			\Lop_e u(e,x) = 0, &\forall\, x \in (0,\delta),\quad \forall\, e\in E,\\
			\sum_{e'\in E} \beta_{e'} u_{e}'(0) = 0,\\
			u(e,\delta) = 0, &\forall\, e\in E, \\
			u(e,0) = u(e',0), &\forall\, e,e'\in E. 
		\end{cases}
	\end{equation}
	Indeed, if $u_1, u_2 $ are two
	solutions of~\eqref{eq_boundary_value_problem}, then $u_1-u_2 $ solves~\eqref{eq_boundary_value_problem_0}.
	Hence, $u=0 $ being the unique solution of~\eqref{eq_boundary_value_problem_0} entails that $u_1=u_2 $. 
	
	The general version of the solution to~\eqref{eq_boundary_value_problem_0} is 
	(see \cite[Section~9]{Fel55})
	\begin{equation}
		\begin{aligned}
			u(e,x) &= A_{e} \s_e(x) + B_{e},
			&\forall\, e&\in E,
			&\forall\, x &>0,
		\end{aligned}  
	\end{equation}
	with $A_e,B_e \in \IR $, for all $e\in E $.
	Thereforem, $u $ is necessarily monotonic on every edge,   
	increasing on $e $, if $A_e\ge 0 $ and decreasing on $e $, if $A_e \le 0 $.
	This means that $u $ attains its maxima and minima necessarily either on $\bv $,
	either on $\partial B(\bv, \delta) $. 
	Let the maxima be located at $\bv $, then $A_e\le 0 $, for all $e\in E $.
	If $A_e=0 $, for all $e\in E $, then $u $ is constant on $B(0,\delta) $.
	If else $\sum_{e\in E} \beta_e \D_x u(e,0)< 0 $, which contradicts the gluing condition.
	Let the minima be located as $\bv $, then  $A_e\ge 0 $, for all $e\in E $. 
	If $A_e=0 $, for all $e\in E $, then $u $ is constant on $B(0,\delta) $.
	If else $\sum_{e\in E} \beta_e \D_x u(e,0)> 0 $, which contradicts the gluing condition.
	Hence minimas and maximas are necessarily located on $\partial B(0,\delta) $.
	From the boundary condition this means that $u $ is trivially null.
	This proves uniqueness.  
\end{proof}

\begin{proposition}
	\label{prop_sticky_estimate}
	For all $\delta>0 $, let $u^{(\delta)} $ be the solution to~\eqref{eq_boundary_value_problem}. 
	We have that
	\begin{equation}
		\begin{aligned}
			\lim_{\delta\to 0} \frac{u^{(\delta)}(\bv)}{\delta} = \rho f(\bv) .
		\end{aligned}
	\end{equation}
\end{proposition}

\begin{proof}
	We have that
	\begin{equation}
		u^{(\delta)}(\bv) = b^{(\delta)}_e(0)  -b^{(\delta)}_e(\delta) - \frac{\rho f(\bv)}{ \s'_e(0)} \braces{\s_e(0) - \s_e(\delta)},
	\end{equation}
	where $b^{(\delta)}_e $ is the function defined as 
	\begin{equation}
		\begin{aligned}
			b^{(\delta)}_e(x) &:= 
			2 \int_{(x,\delta)} \int_{(0,y)} f_{e}(\zeta) \,\m_{e}(\rd \zeta) \, \s_e(\rd y),
			& e&\in E,
			& x&\in [0,\delta),
			& \delta&>0.
		\end{aligned}
	\end{equation}
	We observe that $b^{(\delta)}_e(\delta)=0 $
	and that 
	\begin{align}
		\lim_{\delta \to 0} \frac{b^{(\delta)}_e(0)}{\delta}
		&\le \lim_{\delta \to 0} \frac{1}{\delta} \int_{(x,\delta)} \int_{(0,y)} \abs{f_{e}(\zeta)} \,\m_{e}(\rd \zeta) \, \s_e(\rd y)
		\\ &\le \xnorm{f}{\infty} 
		\lim_{\delta \to 0} \frac{1}{\delta}
		\int_{(x,\delta)} \int_{(0,y)} \,\m_{e}(\rd \zeta) \, \s_e(\rd y)
		\\ &\le \xnorm{f}{\infty} 
		\lim_{\delta \to 0} \braces{\m_{e}((0,\delta)) \frac{\s_e(\delta)-\s_e(0)}{\delta}}  = 0, 
	\end{align}
	for all $e\in E $.
	Hence, 
	\begin{align}
		\lim_{\delta\to 0} \frac{u^{(\delta)}(\bv)}{\delta}
		&= \lim_{\delta\to 0} \frac{\rho f(\bv)}{ \s'_e(0)} \frac{\s_e(\delta) -  \s_e(0)}
		{\delta} = \rho f(\bv). 
	\end{align}
	This completes the proof. 
\end{proof}

\subsection*{Acknowledgments}
This research was supported by the ANR project DREAMeS (ANR-21-CE46-0002-04). 
The author thanks Antoine Lejay and Nabil Kazi-Tani for initial discussions on the subject, and in particular, Antoine Lejay for the valuable reference~\cite{weber2001occupation}.

%========== Bibliography ==========
	\bibliography{bibfile}

\begin{thebibliography}{10}

\bibitem{achdou2019class_infinite}
Y.~Achdou, M.-K. Dao, O.~Ley, and N.~Tchou.
\newblock A class of infinite horizon mean field games on networks.
\newblock {\em Netw. Heterog. Media}, 14(3):537--566, 2019.

\bibitem{Ami}
M.~Amir.
\newblock Sticky {B}rownian motion as the strong limit of a sequence of random
  walks.
\newblock {\em Stochastic Process. Appl.}, 39(2):221--237, 1991.

\bibitem{Anagnostakis2022}
A.~Anagnostakis.
\newblock {Functional convergence to the local time of a sticky diffusion}.
\newblock {\em Electronic Journal of Probability}, 28(88):1--26, 2023.

\bibitem{anagnostakis2023general}
A.~Anagnostakis, A.~Lejay, and D.~Villemonais.
\newblock General diffusion processes as limit of time-space {Markov} chains.
\newblock {\em Ann. Appl. Probab.}, 33(5):3620--3651, 2023.

\bibitem{AnkKruUru2}
S.~Ankirchner, T.~Kruse, and M.~Urusov.
\newblock Wasserstein convergence rates for coin tossing approximations of
  continuous markov processes, 03 2019.
\newblock Preprint arXiv:1903.07880.

\bibitem{ankirchner2020functional}
S.~Ankirchner, T.~Kruse, and M.~Urusov.
\newblock A functional limit theorem for coin tossing {Markov} chains.
\newblock {\em Ann. Inst. Henri Poincar{\'e}, Probab. Stat.}, 56(4):2996--3019,
  2020.

\bibitem{barles2024nonlocal}
G.~Barles, O.~Ley, and E.~Topp.
\newblock Nonlocal {Hamilton}-{Jacobi} {Equations} on a network with
  {Kirchhoff} type conditions.
\newblock Preprint, {arXiv}:2411.13126 [math.{AP}] (2024), 2024.

\bibitem{Barlow1989walsh}
M.~T. Barlow, J.~Pitman, and M.~Yor.
\newblock On {Walsh's} brownian motions.
\newblock {\em S\'eminaire de probabilit\'es de Strasbourg}, 23:275--293, 1989.

\bibitem{bednarz2024diameter}
E.~Bednarz, P.~A. Ernst, and A.~Osekowski.
\newblock On the diameter of the stopped spider process.
\newblock {\em Mathematics of Operations Research}, 49(1):346--365, 2024.

\bibitem{berry2024stationary}
J.~Berry and F.~Camilli.
\newblock Stationary mean field games on networks with sticky transition
  conditions.
\newblock {\em arXiv preprint arXiv:2406.19739}, 2024.

\bibitem{berry2024sticky}
J.~Berry and F.~Colantoni.
\newblock Sticky diffusions on star graphs: characterization and {I}t\^o
  formula.
\newblock {\em arXiv preprint arXiv:2411.05441}, 2024.

\bibitem{bobrowski2024snapping}
A.~Bobrowski and E.~Ratajczyk.
\newblock From snapping out brownian motions to walsh's spider processes on
  star-like graphs.
\newblock {\em arXiv preprint arXiv:2406.16800}, 2024.

\bibitem{borodin2019local}
A.~Borodin and P.~Salminen.
\newblock On the local time process of a skew brownian motion.
\newblock {\em Transactions of the American Mathematical Society},
  372(5):3597--3618, 2019.

\bibitem{BorSal}
A.~N. Borodin and P.~Salminen.
\newblock {\em Handbook of {B}rownian motion---facts and formulae}.
\newblock Probability and its Applications. Birkh{\"a}user Verlag, Basel, 1996.

\bibitem{breiman1992probability}
L.~Breiman.
\newblock {\em Probability}, volume~7 of {\em Class. Appl. Math.}
\newblock Philadelphia, PA: SIAM, 1992.

\bibitem{brooks1982weak}
J.~Brooks and R.~Chacon.
\newblock Weak convergence of diffusions, their speed measures and time
  changes.
\newblock {\em Advances in Mathematics}, 46(2):200--216, 1982.

\bibitem{camilli2024continuousdependence}
F.~Camilli and C.~Marchi.
\newblock A continuous dependence estimate for viscous {Hamilton}-{Jacobi}
  equations on networks with applications.
\newblock {\em Calc. Var. Partial Differ. Equ.}, 63(1):22, 2024.
\newblock Id/No 18.

\bibitem{ccinlar1980semimartingales}
E.~{\c{C}}inlar, J.~Jacod, P.~Protter, and M.~J. Sharpe.
\newblock Semimartingales and markov processes.
\newblock {\em Zeitschrift f{\"u}r Wahrscheinlichkeitstheorie und verwandte
  Gebiete}, 54(2):161--219, 1980.

\bibitem{criens2023feller}
D.~Criens.
\newblock On the feller--dynkin and the martingale property of one-dimensional
  diffusions.
\newblock {\em Electronic Communications in Probability}, 28:1--15, 2023.

\bibitem{Dyn1}
E.~B. Dynkin.
\newblock {\em Markov processes. {V}ols. {I}, {II}}, volume 122 of {\em
  Translated with the authorization and assistance of the author by J. Fabius,
  V. Greenberg, A. Maitra, G. Majone. Die Grundlehren der Mathematischen
  Wissenschaften, B{\"a}nde 121}.
\newblock Academic Press Inc., Publishers, New York; Springer-Verlag,
  Berlin-G{\"o}ttingen-Heidelberg, 1965.

\bibitem{EngPes}
H.-J. Engelbert and G.~Peskir.
\newblock Stochastic differential equations for sticky {B}rownian motion.
\newblock {\em Stochastics}, 86(6):993--1021, 2014.

\bibitem{ethier2005markov}
S.~N. Ethier and T.~G. Kurtz.
\newblock {\em Markov processes. {Characterization} and convergence.}
\newblock Wiley Ser. Probab. Stat. Hoboken, NJ: John Wiley \& Sons, 2005.

\bibitem{Fel55}
W.~{Feller}.
\newblock {On second order differential operators}.
\newblock {\em {Ann. Math. (2)}}, 61:90--105, 1955.

\bibitem{frank1984random}
D.~H. Frank and S.~Durham.
\newblock Random motion on binary trees.
\newblock {\em J. Appl. Probab.}, 21:58--69, 1984.

\bibitem{freidlin2000diffusion}
M.~Freidlin and S.-J. Sheu.
\newblock Diffusion processes on graphs: stochastic differential equations,
  large deviation principle.
\newblock {\em Probability theory and related fields}, 116:181--220, 2000.

\bibitem{freidlin1993diffusion}
M.~I. Freidlin and A.~D. Wentzell.
\newblock Diffusion processes on graphs and the averaging principle.
\newblock {\em The Annals of probability}, pages 2215--2245, 1993.

\bibitem{friedman1994diffusionsnetworks}
A.~Friedman and C.~Huang.
\newblock Diffusion in network.
\newblock {\em J. Math. Anal. Appl.}, 183(2):352--384, 1994.

\bibitem{harrisson1981onskew}
J.~M. Harrison and L.~A. Shepp.
\newblock On skew {Brownian} motion.
\newblock {\em Ann. Probab.}, 9:309--313, 1981.

\bibitem{Ito2006}
K.~It{\^o}.
\newblock {\em Essentials of stochastic processes. {Translated} from the 1957
  {Japanese} original.}, volume 231 of {\em Transl. Math. Monogr.}
\newblock Providence, RI: American Mathematical Society (AMS), 2006.

\bibitem{ItoMcKean96}
K.~{It\^o} and H.~P. jun. {McKean}.
\newblock {\em {Diffusion processes and their sample paths.}}
\newblock Berlin: Springer-Verlag, 1996.

\bibitem{Kal}
O.~Kallenberg.
\newblock {\em Foundations of modern probability. {In} 2 volumes}, volume~99 of
  {\em Probab. Theory Stoch. Model.}
\newblock Cham: Springer, 3rd revised and expanded edition edition, 2021.

\bibitem{lempa2024diffusion}
J.~Lempa, E.~Mordecki, and P.~Salminen.
\newblock Diffusion spiders: Green kernel, excessive functions and optimal
  stopping.
\newblock {\em Stochastic Processes and their Applications}, 167:104229, 2024.

\bibitem{lumer1980connecting}
G.~Lumer.
\newblock Connecting of local operators and evolution equations on networks.
\newblock Potential theory, {Proc}. {Colloq}., {Copenhagen} 1979, {Lect}.
  {Notes} {Math}. 787, 219-234 (1980)., 1980.

\bibitem{mugnolo2019actually}
D.~Mugnolo.
\newblock What is actually a metric graph?
\newblock {\em arXiv preprint arXiv:1912.07549}, 2019.

\bibitem{ohavi2023comparison}
I.~Ohavi.
\newblock Comparison principle for {Walsh}'s spider {HJB} equations with non
  linear local time {Kirchhoff}'s boundary transmission.
\newblock Preprint, {arXiv}:2312.01362 [math.{AP}] (2023), 2023.

\bibitem{ramirez2024stickyLevy}
M.~Ram{\'{\i}}rez and G.~Uribe~Bravo.
\newblock The sticky {L{\'e}vy} process as a solution to a time change
  equation.
\newblock {\em J. Math. Anal. Appl.}, 530(1):18, 2024.
\newblock Id/No 127742.

\bibitem{RevYor}
D.~Revuz and M.~Yor.
\newblock {\em Continuous martingales and {B}rownian motion}, volume 293 of
  {\em Grundlehren der Mathematischen Wissenschaften [Fundamental Principles of
  Mathematical Sciences]}.
\newblock Springer-Verlag, Berlin, third edition, 1999.

\bibitem{RogWilV1}
L.~C.~G. Rogers and D.~Williams.
\newblock {\em Diffusions, {Markov} processes, and martingales. {Vol}. 1:
  {Foundations}.}
\newblock Chichester: Wiley, 2nd ed. edition, 1994.

\bibitem{RogWilV2}
L.~C.~G. Rogers and D.~Williams.
\newblock {\em Diffusions, {M}arkov processes, and martingales. {V}ol. 2}.
\newblock Cambridge Mathematical Library. Cambridge University Press,
  Cambridge, 2000.
\newblock It{\^o} calculus, Reprint of the second (1994) edition.

\bibitem{Sal2017}
M.~{Salins} and K.~{Spiliopoulos}.
\newblock {Markov processes with spatial delay: path space characterization,
  occupation time and properties}.
\newblock {\em {Stoch. Dyn.}}, 17(6):21, 2017.
\newblock Id/No 1750042.

\bibitem{salisbury1986construction}
T.~S. Salisbury.
\newblock Construction of right processes from excursions.
\newblock {\em Probab. Theory Relat. Fields}, 73:351--367, 1986.

\bibitem{salminen2024occupation}
P.~Salminen and D.~Stenlund.
\newblock Occupation times on the legs of a diffusion spider.
\newblock {\em arXiv preprint arXiv:2411.09976}, 2024.

\bibitem{Stein2005}
E.~M. Stein and R.~Shakarchi.
\newblock {\em Real analysis. {Measure} theory, integration, and {Hilbert}
  spaces}, volume~3 of {\em Princeton Lect. Anal.}
\newblock Princeton, NJ: Princeton University Press, 2005.

\bibitem{walsh1978diffusion}
J.~B. Walsh.
\newblock A diffusion with a discontinuous local time.
\newblock {\em Ast{\'e}risque}, 52(53):37--45, 1978.

\bibitem{warren2015stickyflows}
J.~Warren.
\newblock Sticky particles and stochastic flows.
\newblock In {\em S\'eminaire de Probabilit\'es XLVII}, pages 17--35. Cham:
  Springer, 2015.

\bibitem{weber2001occupation}
M.~Weber.
\newblock On occupation time functionals for diffusion processes and
  birth-and-death processes on graphs.
\newblock {\em The Annals of Applied Probability}, 11(2):544--567, 2001.

\end{thebibliography}
	\bibliographystyle{abbrv}

\end{document}